\documentclass[11pt, a4paper]{article}

\usepackage[english]{babel}
\usepackage[a4paper]{geometry}
\usepackage{amsmath,amssymb}
\usepackage{amsfonts}
\usepackage{amsthm}
\usepackage{graphicx}
\usepackage[utf8]{inputenc}
\usepackage{mathtools}
\usepackage{paralist}
\usepackage{booktabs}

\usepackage{color}
\usepackage{tikz}
\usetikzlibrary{matrix}

\numberwithin{equation}{section}

\theoremstyle{plain}
\newtheorem{theorem}{Theorem}[section]

\newtheorem{lemma}[theorem]{Lemma}
\newtheorem{corollary}[theorem]{Corollary}

\theoremstyle{definition}

\newtheorem{remark}[theorem]{Remark}
\newtheorem{example}[theorem]{Example}
\newtheorem{algorithm}[theorem]{Algorithm}
\newcommand{\C}{\mathbb{C}}
\newcommand{\R}{\mathbb{R}}

\newcommand{\bH}{\mathbb{H}}

\newcommand{\cO}{\mathcal{O}}

\newcommand{\cR}{\mathcal{R}}

\newcommand{\coloneq}{\mathrel{\mathop:}=}
\newcommand{\eqcolon}{=\mathrel{\mathop:}}

\newcommand{\eps}{\varepsilon}

\newcommand{\conj}[1]{\overline{#1}}
\newcommand{\comp}[1]{#1^{\operatorname{c}}}

\newcommand{\dd}{{\operatorname{d}}}
\newcommand{\ee}{{\operatorname{e}}}
\newcommand{\ii}{{\operatorname{i}}}

\DeclarePairedDelimiter{\abs}{\lvert}{\rvert}

\DeclarePairedDelimiter{\cc}{[}{]}
\DeclarePairedDelimiter{\oc}{]}{]}
\DeclarePairedDelimiter{\co}{[}{[}
\DeclarePairedDelimiter{\oo}{]}{[}

\DeclareMathOperator{\re}{Re}
\DeclareMathOperator{\im}{Im}

\DeclareMathOperator{\capacity}{cap}

\DeclareMathOperator{\wind}{wind}

\renewcommand{\labelenumi}{\textup{(\roman{enumi})}}

\title{Walsh's Conformal Map onto Lemniscatic Domains for Several Intervals}
\author{Klaus Schiefermayr\footnotemark[1] \and Olivier 
S\`{e}te\footnotemark[2]}
\date{February 11, 2024}

\begin{document}
\maketitle

\renewcommand{\thefootnote}{\fnsymbol{footnote}}

\footnotetext[1]{University of Applied Sciences Upper Austria, Campus Wels, 
Austria, \\ \texttt{klaus.schiefermayr@fh-wels.at}}

\footnotetext[2]{Berlin, Germany.
\texttt{sete@math.tu-berlin.de}, ORCID: 0000-0003-3107-3053}

\renewcommand{\thefootnote}{\arabic{footnote}}

\begin{abstract}
We consider Walsh's conformal map from the complement of a compact set
$E = \cup_{j=1}^\ell E_j$ with $\ell$ components onto a lemniscatic domain 
$\widehat{\C} \setminus L$, where $L$ has the form
$L = \{ w \in \C : \prod_{j=1}^\ell \abs{w - a_j}^{m_j} \leq \capacity(E) \}$.
We prove that the exponents $m_j$ appearing in $L$ satisfy $m_j = \mu_E(E_j)$,
where $\mu_E$ is the equilibrium measure of $E$.
When $E$ is the union of $\ell$ real intervals,
we derive a fast algorithm for computing the centers $a_1, \ldots, a_\ell$.
For $\ell = 2$, the formulas for $m_1, m_2$ and $a_1, a_2$ are explicit.
Moreover, we obtain the conformal map numerically.
Our approach relies on the real and complex Green's functions of $\widehat{\C} \setminus E$ and $\widehat{\C} \setminus L$.
\end{abstract}

\paragraph*{Keywords:}
Conformal map, Lemniscatic domain, Multiply connected domain,
Several intervals, Green's function, Logarithmic capacity,
Equilibrium measure.

\paragraph*{AMS Subject Classification (2020):}
30C35; % General theory of conformal mappings
30C20; % Conformal mappings of special domains
30C85; % (1980–now)Capacity and harmonic measure in the complex plane [See also 31A15]
% 31A15 %(1973–now)Potentials and capacity, harmonic measure, extremal length and related notions in two dimensions [See also 30C85]

\section{Introduction}

When dealing with approximation problems on the interval $[-1,1]$,
it is often convenient to formulate the problem on the unit circle by
using the well-known Joukowsky map $z\mapsto{z}+\sqrt{z^2-1}$,
which maps the complement of the interval $[-1,1]$ onto the complement
of the unit disk, or its inverse $z\mapsto\frac{1}{2}(z+z^{-1})$.
The mapping $z\mapsto{z}+\sqrt{z^2-1}$ is an example of the famous 
Riemann mapping theorem (setting $E:=[-1,1]$) which states that for each
simply connected compact set~$E$ there exists a unique (if suitably 
normalized at $z=\infty$) conformal map from the complement of~$E$ to the 
complement of the unit disk.  J.\,L.\,Walsh~\cite{Walsh1956} found a
rather canonical generalization for the case when $E$ is the union of 
$\ell$~simply connected compact sets. The corresponding conformal 
map~$\Phi$ maps the complement of $E$ onto a so-called lemniscatic 
domain, which is a generalization of a classical lemniscate. 
More precisely, Walsh's theorem reads as follows.

%----- Theorem: Walsh -----%

\begin{theorem}[{\cite[Sect.~III]{Walsh1956}}] \label{thm:walsh_map}
Let $E_1, \ldots, E_\ell \subseteq \C$ be disjoint simply connected, infinite compact sets and let
\begin{equation} \label{eqn:E}
E = \bigcup_{j=1}^\ell E_j,
\end{equation}
that is, $\comp{E} \coloneq \widehat{\C} \setminus E$ is an $\ell$-connected domain. Then there exists a unique compact set~$L$ of the form
\begin{equation} \label{eqn:lemniscatic_domain}
L \coloneq \biggl\{ w \in \C : \prod_{j=1}^\ell \abs{w-a_j}^{m_j} \leq \capacity(E) \biggr\},
\end{equation}
where $a_1,\ldots,a_\ell\in\C$ are pairwise distinct and
$m_1, \ldots, m_\ell > 0$ are real with $\sum_{j=1}^\ell m_j = 1$,
and a unique conformal map
\begin{equation}\label{eqn:Phi}
\Phi : \comp{E} \to \comp{L} \quad \text{with} \quad
\Phi(z) = z + \cO(1/z) \quad \text{at } \infty.
\end{equation}
If $E$ is bounded by Jordan curves, then $\Phi$ extends to a homeomorphism from  $\overline{\comp{E}}$ to $\overline{\comp{L}}$.
\end{theorem}

%----- Text -----%

Let us note that the \emph{centers} $a_1, \ldots, a_\ell$ of $L$ and the
\emph{exponents} $m_1, \ldots, m_\ell$ of $L$ in Theorem~\ref{thm:walsh_map}
are uniquely determined.
The unbounded domain $\comp{L}$ is called a \emph{lemniscatic domain}, see~\cite[p.~106]{Grunsky1978}.  

Theorem~\ref{thm:walsh_map} was first obtained by Walsh in his seminal 1956 
paper~\cite{Walsh1956}.  Further existence proofs were published by 
Grunsky~\cite{Grunsky1957a,Grunsky1957b,Grunsky1978}, 
Jenkins~\cite{Jenkins1958}, and Landau~\cite{Landau1961}.
None of these papers 
contain any explicit example, which might be the reason why Walsh's map has 
not been widely used so far.
However, in~\cite[pp.~374--377]{Walsh2000}, Gaier recognizes this 
conformal map as one of Walsh's major contributions.
The first explicit examples of Walsh's map were derived 
in~\cite{SeteLiesen2016}.
% and applied in~\cite{SeteLiesen2017} for polynomial 
% approximation on disconnected compact sets.
In~\cite{NasserLiesenSete2016}, Nasser, Liesen and the second author obtained a 
numerical method for computing Walsh's map for sets $E$ bounded by 
smooth Jordan curves.  This numerical algorithm 
also yields a method for the numerical computation of the logarithmic capacity 
of compact sets~\cite{LiesenSeteNasser2017}.

Walsh's conformal map is intimately connected with several quantities from 
potential and approximation theory, including the logarithmic capacity 
$\capacity(E)$, 
the equilibrium measure of $E$ (see Theorem~\ref{thm:mj_muE} below),
and the Green's function of $\comp{L}$ and of $\comp{E}$, given 
respectively by the simple expressions
\begin{equation*}
g_L(w) = \sum_{j=1}^\ell m_j \log \abs{w - a_j} - \log(\capacity(E))
\quad \text{and} \quad
g_E(z) = g_L(\Phi(z)).
\end{equation*}
Note that $g_L$ is harmonic in $\C \setminus \{ a_1, \ldots, a_\ell \}$, and 
has the simplest possible form.

An interesting application of Walsh's conformal map in approximation theory is 
the construction of the \emph{Faber-Walsh polynomials} of compact sets $E$ with 
several components as in Theorem~\ref{thm:walsh_map}; see Walsh's original 
article~\cite{Walsh1958}, Suetin's book~\cite{Suetin1998}, and the 
article~\cite{SeteLiesen2017}.
Any function $f$ that is analytic on $E$ can be expanded in a series in the 
Faber-Walsh polynomials on $E$.  The Faber-Walsh polynomials generalize the 
well-known Faber polynomials (defined on simply connected compact sets) and the
Chebyshev polynomials of the first kind (on an interval).
Liesen and the second author considered properties of the Faber-Walsh 
polynomials and gave the first explicit examples in~\cite{SeteLiesen2017}, in 
particular in the case that $E$ is the union of two real intervals or two 
disks.  For two symmetric intervals, the construction is explicit.
In the present paper, we derive a way to numerically compute the 
conformal map for any number of real intervals, which will open up the way for 
polynomial approximation with Faber-Walsh polynomials on sets $E$ consisting of 
several intervals.

For given~$E$ with logarithmic capacity $\capacity(E)$ as in 
Theorem~\ref{thm:walsh_map}, there are three problems to tackle:

1. Find the exponents $m_1,\ldots,m_\ell$.

2. Find the centers $a_1,\ldots,a_\ell$.

3. Find the conformal map $\Phi:\comp{E}\to\comp{L}$.

In our previous work~\cite{SchiefermayrSete2023} and~\cite{SchiefermayrSete-II}, the
focus lies on sets $E$ which are polynomial pre-images of simply connected
compact sets $\Omega$.  This implies that the exponents $m_1, \ldots, m_\ell$
are rational numbers and therefore the set $\partial L$ with $L$ defined
in~\eqref{eqn:lemniscatic_domain} is a classical lemniscate, that is,
the level set of a polynomial.
In this paper, we drop the assumption that $E$ is a polynomial pre-image.
First, we consider the general setting when~$E$ is as in~\eqref{eqn:E} and
second, we consider the case when~$E$ is the union of~$\ell$ real intervals in
more detail.  In both parts, the exponents $m_1,\ldots, m_\ell$ may be
irrational and $\partial L$ is no longer a classical lemniscate.

The paper is organized as follows.
In Section~\ref{sect:equilibrium}, we give a complete solution of Problem~1.
More precisely, we show that $m_j=\mu_E(E_j)$, $j=1,\ldots,\ell$, where $\mu_E$
is the equilibrium measure of~$E$, see Theorem~\ref{thm:mj_muE}.  As a
consequence, together with our previous results, we obtain that if $E$ is a
polynomial pre-image then $\mu_E(E_j)$, $j = 1, \ldots, \ell$, are rational
numbers, see Theorem~\ref{thm:mu_E_rational}.

In Section~\ref{sect:general_results}, we first derive a formula connecting the
exponents $m_1,\ldots,m_{\ell}$ and the centers $a_1,\ldots,a_{\ell}$ with a
certain coefficient in the asymptotic expansion of the Green's function $g_E$,
which is used for computing the centers $a_1,\ldots,a_{\ell}$.  For sets with
$\ell = 2$ components and a symmetry condition, this formula enables us to
derive an explicit expression for $a_1$ and $a_2$ in terms of the Green's
function~$g_E$.

In the last two sections, we focus on the case when $E$ is the union of
$\ell$~real intervals.  In Section~\ref{sect:domain_for_intervals}, we first
recall the formulas for the Green's function $g_E(z)$, the logarithmic
capacity~$\capacity(E)$, and the equilibrium measure~$\mu_E$.
By the results of Section~\ref{sect:equilibrium}, we obtain an explicit
integral formula for the exponents $m_1,\ldots,m_\ell$, see
Theorem~\ref{thm:mj_muE_intervals}.

Concerning the centers, for $\ell = 2$ intervals, we deduce explicit formulas
for $a_1, a_2$ in terms of the endpoints of the intervals,
see Remark~\ref{rem:two_intervals}.  For arbitrary $\ell$, we derive an
algorithm for computing $a_1,\ldots,a_\ell$, see Algorithm~\ref{algo:aj}, which
converges in very few iteration steps to the prescribed tolerance in all our
numerical experiments.

In Section~\ref{sect:map}, we extend the equality of the Green's functions,
$g_E(z) = g_L(w)$, to the \emph{complex} Green's functions, see
Theorem~\ref{thm:Phi_eqn}.  This equation of the complex Green's functions has
a unique solution, see Theorem~\ref{thm:Phi_eqn_unique_sol}, and allows us to
numerically compute $w = \Phi(z)$ for $z \in \comp{E}$, which is illustrated in
several examples.

%----------------------------------------------------------%
\section{The Exponents in Terms of the Equilibrium Measure}
\label{sect:equilibrium}
%----------------------------------------------------------%

Throughout this article, we make use of the \emph{Wirtinger derivatives}
\begin{equation}
\partial_z = \frac{1}{2} (\partial_x - \ii \partial_y) \quad \text{and} \quad
\partial_{\conj{z}} = \frac{1}{2} (\partial_x + \ii \partial_y).
\end{equation}
In particular, if $f$ is analytic then
\begin{equation} \label{eqn:wirtinger}
2 \partial_z \log \abs{f(z)} = \frac{f'(z)}{f(z)} \quad \text{and} \quad
2 \partial_z \re(f(z)) = f'(z).
\end{equation}
Let $g_E$ be the Green's function of $\comp{E}$ with pole at infinity.
The exponents $m_1, \ldots, m_\ell$ in Theorem~\ref{thm:walsh_map} are related 
to the Green's function  by
\begin{equation} \label{eqn:mj_with_integral}
m_j = \frac{1}{2 \pi \ii} \int_{\gamma_j} 2 \partial_z g_E(z) \, \dd z, \quad
j = 1, \ldots, \ell,
\end{equation}
where $\gamma_j$ is a (smooth) closed curve in $\C \setminus E$ such that its 
winding number satisfies $\wind(\gamma; z) = \delta_{jk}$ for $z \in E_k$, 
i.e., $\gamma_j$ surrounds $E_j$ but no other component $E_k$, $k \neq j$;
see~\cite[Thm.~2.3]{SchiefermayrSete2023}.  
In the next theorem, we show that the exponents $m_j$ are also given by the 
\emph{equilibrium measure} of the components of $E$.  For a definition and 
properties of the equilibrium measure, see, e.g., Ransford~\cite{Ransford1995} 
or Garnett and Marshall~\cite{GarnettMarshall2005} (where it is called 
equilibrium distribution).

%----- Theorem: m_j with equilibrium measure -----%

\begin{theorem} \label{thm:mj_muE}
Let $E = \cup_{j=1}^\ell E_j$ be as in Theorem~\ref{thm:walsh_map} and let 
$\mu_E$ be the equilibrium measure of $E$.  Then the exponents $m_j$,
determined in Theorem~\ref{thm:walsh_map}, are given by the equilibrium
measure of $E_j$, that is
\begin{equation}
m_j = \mu_E(E_j), \quad j = 1, \ldots, \ell.
\end{equation}
\end{theorem}

\begin{proof}
Since $\capacity(E) > 0$, the Green's function $g_E$ can be written as
\begin{equation*}
g_E(z) = \int_E \log \abs{z - \zeta} \, \dd \mu_E(\zeta) - \log(\capacity(E)),
\quad z \in \comp{E},
\end{equation*}
see~\cite[p.~85, Thm.~4.4]{GarnettMarshall2005} or~\cite[p.~107]{Ransford1995}.
Taking the Wirtinger derivative and using~\eqref{eqn:wirtinger}, we obtain 
(differentiation and integration can be exchanged)
\begin{equation*}
2 \partial_z g_E(z)
= \int_E 2 \partial_z \log \abs{z - \zeta} \, \dd \mu_E(\zeta)
= \int_E \frac{1}{z - \zeta} \, \dd \mu_E(\zeta).
\end{equation*}
% Let us show that we can indeed differentiate under the integral.  Since 
% $\partial_z = \frac{1}{2} (\partial_x - \ii \partial_y)$, we show this for 
% $\partial_x$ and $\partial_y$.  For each $x$ (or $y$) with $z = x + \ii y \in 
% \C \setminus E$, the function $\zeta \mapsto \log \abs{z - \zeta}$ is 
% continuous, hence Borel-measurable, and bounded on $E$, hence 
% $\mu_E$-integrable.  Moreover, for each $x_0$, the partial derivative 
% $\partial_x \log \abs{z - \zeta} = (x-u) / \abs{z-\zeta}^2$ with $\zeta = u + 
% \ii v$ exists in a neighborhood of $x_0$ and satisfies $\abs{\partial_x \log 
% \abs{z - \zeta}} \leq \abs{z - \zeta}^{-1}$, which is bounded for $x$ in a 
% neighborhood of $x_0$ and all $\zeta \in E$.  Hence, 
% by~\cite[pp.~147--148]{Elstrodt2007}, we can differentiate under the integral.
% The same holds for the derivative $\partial_y$.
% 
Let $\gamma_j$ be a $C^1$-smooth closed curve in $\C \setminus E$ such that its 
winding number satisfies $\wind(\gamma_j; z) = \delta_{jk}$ for $z \in E_k$, 
i.e., $\gamma_j$ surrounds $E_j$ but no $E_k$ with $k \neq j$.  Then, 
by~\eqref{eqn:mj_with_integral}, the exponent $m_j$ is given by
\begin{equation*}
m_j
= \frac{1}{2 \pi \ii} \int_{\gamma_j} \int_E \frac{1}{z - \zeta} \, \dd 
\mu_E(\zeta) \, \dd z
= \int_E \frac{1}{2 \pi \ii} \int_{\gamma_j} \frac{1}{z - \zeta} \, \dd z \, 
\dd \mu_E(\zeta),
\end{equation*}
where the order of integration can be exchanged by Fubini's theorem.
% The latter can be shown as follows.  Let $\gamma_j : \cc{0, 1} \to \C 
% \setminus E$.  Then
% \begin{align*}
% \int_{\gamma_j} \int_E \frac{1}{z - \zeta} \, \dd \mu_E(\zeta) \, \dd z
% = \int_0^1 \int_E \frac{\gamma_j'(t)}{\gamma_j(t) - \zeta} \, \dd \mu_E(\zeta)
% \, \dd t.
% \end{align*}
% The function $f : \cc{0, 1} \times E \to \C$, $f(t,\zeta) = \gamma_j'(t) / 
% (\gamma_j(t) - \zeta)$, is measurable.  Since $f$ is bounded, say $\abs{f(t, 
% \zeta)} \leq C < \infty$, \begin{equation*}
% \int_0^1 \int_E \abs{f(t, \zeta)} \, \dd \mu_E(\zeta) \, \dd t
% \leq \int_0^1 \int_E C \, \dd \mu_E(\zeta) \, \dd t = C < \infty.
% \end{equation*}
% By the Fubini-Tonelli theorem, $f$ is product integrable, hence Fubini's 
% theorem allows to change the order of integration; see, e.g., 
% \cite[p.~175]{Elstrodt2007}.
Since $E = \cup_{j=1}^\ell E_j$, we obtain
\begin{equation*}
m_j
= \sum_{k=1}^\ell \int_{E_k} \frac{1}{2 \pi \ii} \int_{\gamma_j} \frac{1}{z - 
\zeta} \, \dd z \, \dd \mu_E(\zeta)
= \int_{E_j} 1 \, \dd \mu_E(\zeta) = \mu_E(E_j),
\end{equation*}
as claimed.
\end{proof}

As a consequence of Theorem~\ref{thm:mj_muE}, the equilibrium measures 
$\mu_E(E_j)$, $j = 1, \ldots, \ell$, are rational numbers whenever $E$ is a 
polynomial pre-image.  The precise statement is the content of the next 
theorem.

%----- Theorem: muE rational -----%

\begin{theorem}\label{thm:mu_E_rational}
Let $E = \cup_{j=1}^\ell E_j$ be as in Theorem~\ref{thm:walsh_map} and let 
$\mu_E$ be the equilibrium measure of $E$. If $E$ is a polynomial pre-image 
then $\mu_E(E_j)$ is a rational number for $j=1,\ldots,\ell$. More precisely, 
if $E = P^{-1}(\Omega) \coloneq \{ z \in \C : P(z) \in\Omega\}$ with a 
polynomial $P$ of degree $n \geq 1$ and a simply connected infinite compact set 
$\Omega\subseteq \C$, then $\mu_E(E_j) = n_j / n$
where $n_j \in \{ 1, \ldots, n \}$
is the number of zeros of $P(z) - \omega$ in $E_j$ for any $\omega \in \Omega$.
\end{theorem}
\begin{proof}
This follows by combining Theorem~\ref{thm:mj_muE} 
and~\cite[Thm.~3.2]{SchiefermayrSete2023}.
\end{proof}

%----- Text -----%

If $E$ consists of several real intervals, then the condition that $E$ is a polynomial pre-image of $\cc{-1, 1}$ is \emph{equivalent} to the condition that $\mu_E(E_1), \ldots, \mu_E(E_\ell)$ are rational; see Theorem~\ref{thm:gE_intervals}\,\ref{it:mu_E_rational}. By Theorem~\ref{thm:mu_E_rational}, one direction holds for more general compact sets.  To the authors' knowledge, whether the other direction can also be shown is an open question.

%------------------------------------------------%
\section{The Lemniscatic Domain for General Sets}
\label{sect:general_results}
%------------------------------------------------%

%----- Text -----%

Let us first discuss the Green's functions of $\comp{E}$ and $\comp{L}$,
where $E$ and $L$ are as in Theorem~\ref{thm:walsh_map}.  The Green's function
(with pole at~$\infty$) of $\comp{L}$ is
\begin{equation} \label{eqn:gL}
g_L(w) = \sum_{j=1}^\ell m_j \log \abs{w-a_j} - \log(\capacity(E)).
\end{equation}
The critical points of $g_L$ are the zeros of
\begin{equation} \label{eqn:gL_wirtinger}
2 \partial_w g_L(w) = \sum_{j=1}^\ell \frac{m_j}{w - a_j},
\end{equation}
and hence are the solutions of the equation
\begin{equation} \label{eqn:poly_crit_pts_gL}
\sum_{j=1}^\ell m_j \prod_{\substack{i=1 \\ i \neq j}}^\ell (w-a_i) = 0,
\end{equation}
where the left-hand side is a monic polynomial of degree~$\ell - 1$.  In particular, $g_L$ has
$\ell-1$ critical points (counted with multiplicity), denoted by
$w_1, \ldots, w_{\ell-1}$, which are located in $\C \setminus L$,
see~\cite[Thm.~2.5]{SchiefermayrSete2023} or~\cite[pp.~67--68]{Walsh1969}.
The Green's function $g_E$ of $\comp{E}$ is related to $g_L$ by
\begin{equation} \label{eqn:gE_gL}
g_E(z) = g_L(w) = g_L(\Phi(z)), \quad z \in \comp{E},
\end{equation}
with the conformal map $\Phi$ from Theorem~\ref{thm:walsh_map};
see~\cite[Sect.~III]{Walsh1958} or~\cite[Sect.~2]{SchiefermayrSete2023}.  In 
particular, $\capacity(E) = \capacity(L)$.
Let $z_1, \ldots, z_{\ell-1} \in \C \setminus E$ denote the critical points
of $g_E$, then $w_j = \Phi(z_j)$, $j = 1, \ldots, \ell-1$, with a suitable
labeling of the critical points, which follows from~\eqref{eqn:gE_gL}, see 
also~\cite[Lem.~2.1]{SchiefermayrSete2023}.
We thus have
\begin{equation} \label{eqn:gL_gE_at_critical_pts}
g_L(w_j)=g_E(z_j), \quad j = 1, \ldots, \ell-1.
\end{equation}

It is well known~\cite[Thm.~5.2.1]{Ransford1995} that the Green's function has the asymptotic behavior
\begin{equation} \label{eqn:gE_asymptotic_with_capacity}
g_E(z) = \log \abs{z} - \log(\capacity(E)) + o(1) \quad \text{at } \infty.
\end{equation}
Next, we obtain a relation between the parameters $a_j, m_j$ and the coefficient of $z^{-2}$ in the Laurent series of $\partial_z g_E$; see~Theorem~\ref{thm:sum_mj_aj}. For this purpose, let us consider the asymptotic behaviour of $g_E$ and $\partial_z g_E$ in more detail.

%----- Lemma: Expansion of Green function -----%

\begin{lemma} \label{lem:expansion_gE}
Let $E$ be as in Theorem~\ref{thm:walsh_map} and let $g_E$ be the Green's 
function of $\comp{E}$.  Then the function $\partial_z g_E$ is analytic in $\C 
\setminus E$ with
\begin{equation} \label{eqn:wirtinger_derivative_gE}
2 \partial_z g_E(z) = \frac{1}{z} + \frac{\alpha}{z^2} + \cO \left( 
\frac{1}{z^3} \right) \quad \text{at } \infty,
\end{equation}
where $\alpha\in\C$, and the Green's function has the expansion
\begin{equation} \label{eqn:gE_asymtptotic}
g_E(z) = \log \abs{z} - \log(\capacity(E)) - \re \left( \frac{\alpha}{z} \right) 
+ \cO \left( \frac{1}{z^2} \right) \quad \text{at } \infty.
\end{equation}
\end{lemma}
\begin{proof}
Since $g_E$ is harmonic in $\C \setminus E$ and 
$\partial_{\conj{z}} \partial_z g_E = \frac{1}{4} \Delta g_E = 0$, the function $\partial_z g_E$ is analytic in $\C \setminus E$. By~\cite[Lem.~2.1 and~(2.10)]{SchiefermayrSete2023},
\begin{equation} \label{eqn:wirtinger_gE_gL}
2 \partial_z g_E(z) = 2 \partial_w g_L(\Phi(z)) \cdot \Phi'(z)
= \sum_{j=1}^\ell \frac{m_j}{\Phi(z) - a_j} \Phi'(z), \quad z \in \C \setminus E.
\end{equation}
Since $\Phi(z) = z + \cO(1/z)$ and thus $\Phi'(z) = 1 + \cO(1/z^2)$, we obtain that
\begin{equation*}
\lim_{z \to \infty} 2 \partial_z g_E(z) = 0 \quad \text{and} \quad
\lim_{z \to \infty} z 2 \partial_z g_E(z) = \sum_{j=1}^\ell m_j = 1,
\end{equation*}
which completes the proof of~\eqref{eqn:wirtinger_derivative_gE}.  Let $h_E$
denote a harmonic conjugate of $g_E$, so that $f(z) = g_E(z) + \ii h_E(z)$ is
analytic (note that $h_E$ and $f$ are in general multi-valued) and
$g_E(z) = \re(f(z))$.  Then, by~\eqref{eqn:wirtinger}, 
$f'(z) = 2 \partial_z g_E(z)$.  Integrating~\eqref{eqn:wirtinger_derivative_gE} yields
$f(z) = \log(z) + c - \alpha/z + \cO(1/z^2)$.  Taking the real part and
recalling~\eqref{eqn:gE_asymptotic_with_capacity} yields~\eqref{eqn:gE_asymtptotic}.
\end{proof}

%----- Theorem: General Formula for Exponents m_j -----%

\begin{theorem} \label{thm:sum_mj_aj}
In the notation of Theorem~\ref{thm:walsh_map}, let $g_E$ be the Green's function of $\comp{E}$ and let $\alpha$ be the coefficient of $1/z^2$ in the Laurent series of $2 \partial_z g_E$ at infinity, see~\eqref{eqn:wirtinger_derivative_gE}. Then
\begin{equation} \label{eqn:sum_mj_aj}
\sum_{j=1}^\ell m_j a_j = \alpha.
\end{equation}
\end{theorem}

\begin{proof}
Let $\Gamma$ be a closed curve in $\C \setminus E$ that contains $E$ in its  
interior, i.e., $\wind(\Gamma; z) = 1$ for all $z \in E$, then $\Phi(\Gamma)$ 
contains $L$ in its interior and, by~\cite[Thm.~2.3]{SchiefermayrSete2023},
\begin{equation*}
\sum_{j=1}^\ell m_j a_j
= \frac{1}{2 \pi \ii} \int_{\Phi(\Gamma)} w 2 \partial_w g_L(w) \, \dd w
= \frac{1}{2 \pi \ii} \int_\Gamma \Phi(z) 2 \partial_z g_E(z) \, \dd z.
\end{equation*}
Since $\Phi(z) 2 \partial_z g_E(z)$ is analytic in $\C \setminus E$ and since  
$\Gamma$ contains $E$ in its interior, the integral is equal to the coefficient 
of $1/z$ in the Laurent series of $\Phi(z) 2 \partial_z g_E(z)$ at infinity.
By~\eqref{eqn:Phi} and~\eqref{eqn:wirtinger_derivative_gE},
\begin{equation*}
\Phi(z) 2 \partial_z g_E(z) = \bigl( z + \cO(z^{-1}) \bigr)
\bigl( z^{-1} + \alpha z^{-2} + \cO(z^{-3}) \bigr)
= 1 + \alpha z^{-1} + \cO(z^{-2}),
\end{equation*}
which completes the proof.
\end{proof}

%----- Text -----%

Theorem~\ref{thm:sum_mj_aj} has two remarkable consequences:
Corollary~\ref{cor:sum_mj_aj_poly_preimage} and Theorem~\ref{thm:sum_mj_aj_ell_intervals}.

\begin{corollary} \label{cor:sum_mj_aj_poly_preimage}
If $E$ is a polynomial pre-image, that is, $E = P^{-1}(\Omega)$, where
$P(z) = \sum_{k=0}^n p_k z^k$ with $p_n \neq 0$ is a polynomial of degree $n \geq 2$,
and $\Omega \subseteq \C$ is a simply connected, infinite compact set,
then $\alpha$ defined in Lemma~\ref{lem:expansion_gE} is given by
\begin{equation} \label{eqn:sum_mj_aj_poly_preimage}
\sum_{j=1}^\ell m_j a_j = \alpha = - \frac{p_{n-1}}{n p_n}.
\end{equation}
\end{corollary}

The proof of Corollary~\ref{cor:sum_mj_aj_poly_preimage} provides an alternative proof of~\cite[Eq.~(3.21)]{SchiefermayrSete2023} and is given in Appendix~\ref{sect:appendix}.

From Theorem~\ref{thm:sum_mj_aj}, for the simplest case $\ell = 2$, we obtain explicit formulas for the centers $a_1$ and $a_2$ if $E_1^* = E_1$ and $E_2^* = E_2$, where $S^* = \{ \conj{z} : z \in S \}$ is the reflection on the real line of a set $S \subseteq \C$.
Note that if $E_j^* = E_j$, then $E_j \cap \R$ is a point or an interval.
If all components $E_j$ satisfy $E_j^* = E_j$, we always label them from left
to right as in~\cite[p.~495]{SchiefermayrSete2023}
and~\cite[Sect.~2.1]{SchiefermayrSete-II}.

%----- Theorem: Formula for a_1, a_2 (l=2) -----%

\begin{theorem} \label{thm:a1a2}
Let $E = E_1 \cup E_2$ and let $z_1 \in \C \setminus E$ be the critical point of $g_E$. Then
\begin{equation} \label{eqn:absa2a1}
\abs{a_2 - a_1} = \frac{\capacity(E)}{m_1^{m_1} m_2^{m_2}} \exp(g_E(z_1))\eqcolon \beta.
\end{equation}
If $E_1^* = E_1$ and $E_2^* = E_2$ then $a_1, a_2 \in \R$ with $a_1 < a_2$ and
\begin{equation} \label{eqn:a2a1_explicit}
a_1 = \alpha - m_2 \beta, \quad a_2 = \alpha + m_1 \beta,
\end{equation}
where $\alpha$ is given in~\eqref{eqn:wirtinger_derivative_gE}.
\end{theorem}
\begin{proof}
Since $E$ and thus $L$ consist of two components, both Green's functions $g_E$ and $g_L$ have a unique critical point $z_1$ and $w_1$, respectively, and both critical points are simple; see~\cite[Thm.~2.5]{SchiefermayrSete2023}. 
By~\eqref{eqn:poly_crit_pts_gL},
\begin{equation*}
w_1 = m_2 a_1 + m_1 a_2.
\end{equation*}
Using $g_E(z_1) = g_L(w_1)$ from~\eqref{eqn:gL_gE_at_critical_pts}, $m_1 + m_2 = 1$, and~\eqref{eqn:gL}, we obtain
\begin{align*}
g_E(z_1) &= g_L(w_1)
= \log \left( \abs{m_1 (a_2-a_1)}^{m_1} \abs{m_2 (a_2-a_1)}^{m_2} \right) - \log(\capacity(E)) \\
&= \log \left( \frac{m_1^{m_1} m_2^{m_2}}{\capacity(E)} \abs{a_2 - a_1} \right),
\end{align*}
which establishes~\eqref{eqn:absa2a1}.  If $E_1^* = E_1$, $E_2^* = E_2$, then $a_1 < a_2$ by~\cite[Thm.~2.8]{SchiefermayrSete2023}, therefore $a_2-a_1 =\beta$.  Combining this with~\eqref{eqn:sum_mj_aj} yields~\eqref{eqn:a2a1_explicit}.
\end{proof}

%----- Remark -----%

\begin{remark}
For $\ell = 3$, the critical points $w_1, w_2$ of $g_L$ are the solutions of the quadratic equation~\eqref{eqn:poly_crit_pts_gL} and can therefore be computed explicitly as functions of $a_1,a_2,a_3$, i.e., $w_j = w_j(a_1, a_2, a_3)$ for $j = 1, 2$. Let $z_1, z_2$ be the corresponding critical points of $g_E$, then, by~\eqref{eqn:gL_gE_at_critical_pts} and~\eqref{eqn:sum_mj_aj}, we obtain the non-linear system of equations
\begin{equation} \label{eqn:sys_a1a2a3}
\begin{aligned}
m_1 \log \abs{w_1 - a_1} + m_2 \log \abs{w_1 - a_2} + m_3 \log \abs{w_1 - a_3} - \log(\capacity(E)) &= g_E(z_1) \\
m_1 \log \abs{w_2 - a_1} + m_2 \log \abs{w_2 - a_2} + m_3 \log \abs{w_2 - a_3} - \log(\capacity(E)) &= g_E(z_2) \\
m_1 a_1 + m_2 a_2 + m_3 a_3 &= \alpha,
\end{aligned}
\end{equation}
which can be solved numerically for $a_1, a_2, a_3$.
\end{remark}

For $\ell \geq 4$, it is not practical to obtain the critical points 
$w_1,\ldots, w_{\ell-1}$ explicitly as zeros of the polynomial 
in~\eqref{eqn:poly_crit_pts_gL}, which is of degree $\ell-1\geq3$.  In the next 
section, we will therefore derive an algorithm for the computation of
$a_1, \ldots, a_\ell$ in the case that $E$ consists of $\ell$ real intervals with arbitrary $\ell$, see Algorithm~\ref{algo:aj}.

%-------------------------------------------------%
\section{The Lemniscatic Domain for Several Intervals}
\label{sect:domain_for_intervals}
%-------------------------------------------------%

Let $E$ be the union of $\ell$ real intervals, $\ell \geq 1$, i.e.,
\begin{equation} \label{eqn:E_l_intervals}
E = \bigcup_{j=1}^\ell \cc{b_{2j-1}, b_{2j}}
= \cc{b_1, b_2} \cup \cc{b_3, b_4} \cup \ldots \cup \cc{b_{2\ell-1}, b_{2\ell}}
\end{equation}
with $b_1 < b_2 < \ldots < b_{2 \ell}$.
The set $\R \setminus E$ consists of the $\ell + 1$ open intervals
\begin{equation} \label{eqn:gaps}
I_0 \coloneq \oo{-\infty, b_1}, \quad I_j \coloneq \oo{b_{2j}, b_{2j+1}}
\text{ with } j = 1, \ldots, \ell-1, \quad I_\ell \coloneq \oo{b_{2\ell}, +\infty}.
\end{equation}
In what follows, the square root of the polynomial
\begin{equation} \label{eqn:H}
H(z) \coloneq \prod_{j=1}^{2 \ell} (z - b_j)
\end{equation}
plays a central role.  Therefore, let us examine the function $\sqrt{H(z)}$ for $z \in \comp{E}$ in detail.

\begin{lemma} \label{lem:sqrtH}
Let $E$ and $H$ be as in~\eqref{eqn:E_l_intervals} and~\eqref{eqn:H}.
Then, the branch of the square root such that $\sqrt{H(z)}$ behaves as 
$z^\ell$ at infinity is given by
\begin{equation}
\sqrt{H(z)} = \prod_{j=1}^{2 \ell} \sqrt{z - b_j}, \quad z \in \C \setminus E,
\end{equation}
with the principal branches of the square roots $\sqrt{z - b_j}$,
and has the following behaviour on the real line:
\begin{equation} \label{eqn:sqrtHx_real}
\sqrt{H(x)} =
(-1)^{\ell - j} \sqrt{\abs{H(x)}}, \quad x \in I_j, \quad j = 0, \ldots, \ell,
\end{equation}
and
\begin{equation}
\lim_{y \searrow 0} \sqrt{H(x \pm \ii y)} = \pm \ii (-1)^{\ell - j} 
\sqrt{\abs{H(x)}}, \quad x \in \cc{b_{2j-1}, b_{2j}}, \quad j = 1, \ldots, \ell,
\end{equation}
with the positive real root of $\sqrt{\abs{H(x)}}$.
\end{lemma}

\begin{proof}
For $z > b_{2 \ell}$, the branch of $\sqrt{H(z)}$ that behaves like $z^\ell$ at 
infinity is given by $\sqrt{H(z)} = \prod_{j=1}^{2 \ell} \sqrt{z - b_j}$ with 
the principal branch of the square roots $\sqrt{z - b_j}$, which extends to $\C 
\setminus \oc{-\infty, b_{2 \ell}}$ by the identity principle.
For the principal branch of the square root, we have the limits
\begin{equation*}
\lim_{y \searrow 0} \sqrt{x \pm \ii y - b_j} =
\begin{cases}
\sqrt{\abs{x - b_j}}, & x \geq b_j, \\
\pm \ii \sqrt{\abs{x - b_j}}, & x < b_j,
\end{cases}
\end{equation*}
with the positive real roots $\sqrt{\abs{x - b_j}}$.
The assertion follows from this observation by starting with $x > b_{2\ell}$ 
and moving with $x$ on the real axis from right to left.
\end{proof}

In the following theorem, we recall the well-known formula for the Green's function of~$\comp{E}$ which seems to first appear in Widom's seminal 1969 paper~\cite[Sect.~14]{Widom1969}, see also Shen, Strang and Wathen~\cite[Sect.~3]{ShenStrangWathen2001}, as well as Peherstorfer~\cite[Sect.~2]{Peherstorfer1990}.

%----- Theorem: Green's function for several intervals -----%

\begin{theorem} \label{thm:gE_intervals}
Let $E$ and $H$ be as in~\eqref{eqn:E_l_intervals} and~\eqref{eqn:H}, 
respectively.
\begin{enumerate}
\item \label{it:R}
There exists a real polynomial
\begin{equation} \label{eqn:R}
R(z) \coloneq \prod_{k=1}^{\ell-1} (z - z_k) = z^{\ell-1} + \sum_{j=0}^{\ell-2} 
r_j z^j
\end{equation}
which is uniquely determined by
\begin{equation} \label{eqn:lgs_for_R}
\int_{b_{2j}}^{b_{2j+1}} \frac{R(x)}{\sqrt{H(x)}} \, \dd x = 0, \quad j = 1, 
\ldots, \ell-1.
\end{equation}
Since $H(x) \geq 0$ for $x \in \cc{b_{2j}, b_{2j+1}}$, $j = 1, \ldots, \ell-1$, 
the positive square root is taken.

\item \label{it:green_L_intervals}
The Green's function $g_E$ of $\comp{E}$ is given by
\begin{equation} \label{eqn:gE_l_intervals}
g_E(z) = \re \biggl( \int_b^z \frac{R(\zeta)}{\sqrt{H(\zeta)}} \, \dd \zeta 
\biggr).
\end{equation}
The integration in~\eqref{eqn:gE_l_intervals} is performed along a path in $\C 
\setminus E$ from any $b \in \{ b_1, \ldots, b_{2 \ell} \}$ to $z \in \C 
\setminus E$, and the branch of $\sqrt{H(\zeta)}$ is as in
Lemma~\ref{lem:sqrtH}.

\item \label{it:equilibrium_measure_L_intervals}
The equilibrium measure of $E$, denoted by $\mu_E$, satisfies
\begin{equation}
\dd \mu_E(x) = 
\frac{1}{\pi} \frac{\abs{R(x)}}{\sqrt{\abs{H(x)}}} \, \dd x, \quad x \in E,
\end{equation}
with the positive real square root in $\sqrt{\abs{H(x)}}$, and thus
\begin{equation} \label{eqn:equilibrium_measure_one_interval}
\mu_E(\cc{b_{2j-1}, b_{2j}})
= \frac{1}{\pi} \int_{b_{2j-1}}^{b_{2j}} \frac{\abs{R(x)}}{\sqrt{\abs{H(x)}}} 
\, \dd x, \quad j = 1, \ldots, \ell.
\end{equation}

\item \label{it:mu_E_rational}
The set $E$ is a polynomial pre-image, that is, $E = P^{-1}(\cc{-1, 1})$ with a polynomial $P$ of degree $n$, if and only if $\mu_E(\cc{b_{2j-1}, b_{2j}}) = n_j/n$ for $j = 1, \ldots, \ell$ with $n_j \in \{ 1,\ldots,n \}$ and $\sum_{j=1}^n n_j = 1$.
\end{enumerate}
\end{theorem}

\begin{proof}
For~\ref{it:R} and~\ref{it:green_L_intervals}, see~\cite[Sect.~14]{Widom1969}, 
\cite[Sect.~3]{ShenStrangWathen2001}, or~\cite[Sect.~2]{Peherstorfer1990}.
For~\ref{it:equilibrium_measure_L_intervals}, 
see~\cite[Lem.~2.2\,(a)]{Peherstorfer1990}, and 
also~\cite[Thm.~8]{ShenStrangWathen2001}.
\ref{it:mu_E_rational} ``$\Rightarrow$'':  Follows immediately from 
Theorem~\ref{thm:mu_E_rational} with $\Omega = \cc{-1, 1}$; see 
also~\cite[Cor.~2.4]{Peherstorfer1991}. ``$\Leftarrow$'':  
See~\cite[Thm.~2.5]{Peherstorfer1993}.
\end{proof}

We typically use the variable $x$ for integration along the real line and 
$\zeta$ or $z$ for integration along a path in the complex plane.
The integrals in~\eqref{eqn:equilibrium_measure_one_interval} are also
called \emph{harmonic frequencies}~\cite[Sect.~3]{Mantica2013}.
In~\cite[Sect.~4]{Mantica2013}, Mantica discusses the integration of
$R(x)/\sqrt{H(x)}$ via scaling and Gauss quadrature with respect
to the Chebyshev measure.

\begin{corollary} \label{cor:intervals}
Let the notation be as in Theorem~\ref{thm:gE_intervals}.
\begin{enumerate}
\item \label{it:coefficients_R}
The coefficients $r_0, \ldots, r_{\ell-2}$ of the polynomial $R$ 
in~\eqref{eqn:R} are the unique solution of the linear algebraic system
\begin{equation} \label{eqn:coeff_R}
\sum_{k=0}^{\ell-1} r_k \int_{b_{2j}}^{b_{2j+1}} \frac{x^k}{\sqrt{H(x)}} \, \dd 
x = 0, \quad j = 1, \ldots, \ell-1,
\end{equation}
where $r_{\ell-1} = 1$.

\item \label{it:crit_pts_green}
The zeros $z_1, \ldots, z_{\ell-1}$ of $R$ are exactly the critical points of 
$g_E$ and therefore satisfy
\begin{equation} \label{eqn:inequality_zi_bj}
b_1 < b_2 < z_1 < b_3 < b_4 < z_2 < b_5 \ldots < z_{\ell-1} < b_{2\ell-1} < 
b_{2 \ell}.
\end{equation}
In particular,
\begin{equation} \label{eqn:sign_R}
R(x) = (-1)^{\ell - j} \abs{R(x)}, \quad x \in \cc{b_{2j-1}, b_{2j}}, \quad
j = 1, \ldots, \ell.
\end{equation}

\item \label{it:green_at_crit_pts}
For $j=1,\ldots,\ell-1$, we have
\begin{equation}
g_E(z_j)=\int_{b_{2j}}^{z_j} \frac{(-1)^{\ell-j}R(x)}{\sqrt{\abs{H(x)}}}\,\dd x
=-\int_{z_j}^{b_{2j+1}}\frac{(-1)^{\ell-j}R(x)}{\sqrt{\abs{H(x)}}}\,\dd x
\end{equation}
with the positive real root in $\sqrt{\abs{H(x)}}$, from which we obtain
\begin{equation}
g_E(z_j) = \frac{1}{2} \int_{b_{2j}}^{b_{2j+1}} \frac{\abs{R(x)}}{\sqrt{\abs{H(x)}}} \, \dd x.
\end{equation}

\item \label{it:cap_L_intervals}
The logarithmic capacity $\capacity(E)$ is given by
\begin{align}
\capacity(E)
&= (b_{2\ell} - \beta_1) \exp \biggl( \int_{b_{2\ell}}^{\infty} \biggl( 
\frac{1}{x - \beta_1} - \frac{R(x)}{\sqrt{H(x)}} \biggr) \, \dd x \biggr) 
\label{eqn:cap_L_intervals_1} \\
&= (\beta_2 - b_1) \exp \biggl( \int_{-\infty}^{b_1} \biggl( 
\frac{R(x)}{\sqrt{H(x)}} - \frac{1}{x - \beta_2} \biggr) \, \dd x \biggr) 
\label{eqn:cap_L_intervals_2}
\end{align}
for any $\beta_1 < b_{2\ell}$ and any $\beta_2 > b_1$, e.g., $\beta_1 = 
b_{2\ell} - 1$ and $\beta_2 = b_1 + 1$.
\end{enumerate}
\end{corollary}

\begin{proof}
\ref{it:coefficients_R} follows immediately from~\eqref{eqn:lgs_for_R}; 
see~\cite[pp.~225--226]{Widom1969} or~\cite[Thm.~3]{ShenStrangWathen2001}.

\ref{it:crit_pts_green}
We get $2 \partial_z g_E(z) = R(z) / \sqrt{H(z)}$ 
from~\eqref{eqn:gE_l_intervals} and~\eqref{eqn:wirtinger}, hence the critical 
points of $g_E$ are precisely the zeros of $R$.  
Equation~\eqref{eqn:inequality_zi_bj}
is a conclusion of the behaviour of $g_E$ on the real 
line; see~\cite[p.~226]{Widom1969} or~\cite[p.~74]{ShenStrangWathen2001}.
A proof for more general sets of the form $E = E_1 \cup \ldots \cup E_\ell$
with $E_j^* = E_j$ is given in~\cite[Thm.~2.8\,(i)]{SchiefermayrSete2023}.

\ref{it:green_at_crit_pts} follows from~\eqref{eqn:gE_l_intervals} and~\eqref{eqn:sqrtHx_real},
see also Widom~\cite[pp.~226--227]{Widom1969}.

\ref{it:cap_L_intervals} follows from $\log(\capacity(E)) = \lim_{z \to \infty} 
(\log \abs{z} - g_E(z))$, where the limit can be taken along the real line, 
either $z \to +\infty$ or $z \to -\infty$.
Let us consider the first formula.  By~\eqref{eqn:gE_l_intervals} and for any 
$\beta_1 < b_{2\ell}$, we then have
\begin{equation*}
\log(\capacity(E)) = \ln(b_{2 \ell} - \beta_1) - \int_{b_{2 \ell}}^\infty 
\biggl( \frac{R(x)}{\sqrt{H(x)}} - \frac{1}{x - \beta_1} \biggr) \, \dd x
\end{equation*}
which implies~\eqref{eqn:cap_L_intervals_1}.
Analogously, formula~\eqref{eqn:cap_L_intervals_2} is proved.
Note that~\eqref{eqn:cap_L_intervals_1} has also been derived 
in~\cite[Eqn.~(16)]{DubininKarp2011} in the case $b_{2\ell} = 1$ and $\beta_1 = 
0$, and~\eqref{eqn:cap_L_intervals_2} has also been derived 
in~\cite[p.~226]{Widom1969} for $\beta_2 = b_1 + 1$.
\end{proof}

The logarithmic capacity of a union of real intervals can also be computed using
theta functions~\cite{BogatyrevGrigoriev2017}
or via Schwarz-Christoffel maps~\cite[p.~751]{EmbreeTrefethen1999}.

%----- Theorem: Exponents with equilibrium measure (interval case) -----%

The next result is a conclusion of Theorem~\ref{thm:mj_muE} and Theorem~\ref{thm:gE_intervals}\,\ref{it:equilibrium_measure_L_intervals}.
% which we state as a theorem because of its importance.
Since we consider the result of great importance, we formulate it rather as
a theorem than as a corollary.
In Appendix~\ref{sect:appendix}, we give an alternative proof,
which uses the representations~\eqref{eqn:gE_l_intervals} 
and~\eqref{eqn:equilibrium_measure_one_interval}
of $g_E$ and $\mu_E(\cc{b_{2j-1}, b_{2j}})$, respectively.

\begin{theorem}\label{thm:mj_muE_intervals}
Let $E$ be as in~\eqref{eqn:E_l_intervals} and let $\mu_E$ be the equilibrium 
measure of $E$, then the exponents $m_1, \ldots, m_\ell$ in the lemniscatic 
domain corresponding to $E$ are given by
\begin{equation} \label{eqn:mj_by_muE_for_intervals}
m_j = \mu_E(\cc{b_{2j-1}, b_{2j}})
= \frac{1}{\pi} \int_{b_{2j-1}}^{b_{2j}} \frac{\abs{R(x)}}{\sqrt{\abs{H(x)}}} 
\, \dd x, \quad j = 1, \ldots, \ell.
\end{equation}
\end{theorem}

%----- Text -----%

If $E$ is the union of $\ell$ real intervals, we can express the coefficient 
$\alpha$ in~\eqref{eqn:wirtinger_derivative_gE} by the endpoints of the 
intervals and the critical points of $g_E$.

%----- Theorem: sum m_j a_j -----%

\begin{theorem} \label{thm:sum_mj_aj_ell_intervals}
Let $E$ be as in~\eqref{eqn:E_l_intervals} and let $z_1, 
\ldots, z_{\ell-1}$ be the critical points of the Green's function $g_E$, then
\begin{equation} \label{eqn:sum_mj_aj_ell_intervals}
\sum_{j=1}^\ell m_j a_j = \alpha
= \frac{1}{2} \sum_{j=1}^{2 \ell} b_j - \sum_{j=1}^{\ell-1} z_j.
\end{equation}
\end{theorem}

\begin{proof}
First, let us recall that the critical points $z_1, \ldots, z_{\ell-1}$ are 
located in the $\ell-1$ gaps of $E$, see~\eqref{eqn:inequality_zi_bj} 
or~\cite[Thm.~2.8\,(i)]{SchiefermayrSete2023}.  By 
formula~\eqref{eqn:gE_l_intervals} and~\eqref{eqn:wirtinger}, we obtain
\begin{align*}
2 \partial_z g_E(z) &= \frac{\prod_{j=1}^{\ell-1} (z - 
z_j)}{\sqrt{\prod_{j=1}^{2\ell} (z - b_j)}}
= \frac{z^{\ell-1} - z^{\ell-2} \sum_{j=1}^{\ell-1} z_j + 
\cO(z^{\ell-3})}{\sqrt{z^{2\ell} - z^{2 \ell-1} \sum_{j=1}^{2 \ell} b_j + 
\cO(z^{2 \ell-2})}} \\
&= \frac{z^{\ell-1} - z^{\ell-2} \sum_{j=1}^{\ell-1} z_j + 
\cO(z^{\ell-3})}{z^{\ell} \sqrt{1 - z^{-1} \sum_{j=1}^{2 \ell} b_j + 
\cO(z^{-2})}} = (*)
\end{align*}
Since $1 / \sqrt{1 - x} = 1 + x/2 + \cO(x^2)$, we further obtain
\begin{align*}
(*) &= \biggl( z^{-1} - z^{-2} \sum_{j=1}^{\ell-1} z_j + \cO(z^{-3}) \biggr)
\biggl( 1 + z^{-1} \frac{1}{2} \sum_{j=1}^{2 \ell} b_j + \cO(z^{-2}) \biggr) \\
&= z^{-1} + z^{-2} \biggl( \frac{1}{2} \sum_{j=1}^{2 \ell} b_j - 
\sum_{j=1}^{\ell-1} z_j \biggr) + \cO(z^{-3}).
\end{align*}
The assertion now follows from Theorem~\ref{thm:sum_mj_aj}  
and~\eqref{eqn:wirtinger_derivative_gE}.
\end{proof}

%----- Remark -----%

\begin{remark} \label{rem:two_intervals}
Consider the union of $\ell = 2$ intervals,
\begin{equation*}
E = \cc{b_1, b_2} \cup \cc{b_3, b_4}, \quad b_1 < b_2 < b_3 < b_4,
\end{equation*}
and let $H(z) = (z-b_1) (z-b_2) (z-b_3) (z-b_4)$ be as in~\eqref{eqn:H}.
By Theorem~\ref{thm:gE_intervals}\,\ref{it:R}, the polynomial $R$ has the form 
$R(z) = z - z_1$, where
\begin{equation*}
z_1 = \int_{b_2}^{b_3} \frac{x}{\sqrt{H(x)}} \, \dd x \Bigg/ \int_{b_2}^{b_3} 
\frac{1}{\sqrt{H(x)}} \, \dd x.
\end{equation*}
By Corollary~\ref{cor:intervals}, the logarithmic capacity of $E$ is
\begin{equation*}
\capacity(E)
= \exp \biggl( \int_{b_4}^{\infty} \biggl( 
\frac{1}{x - b_4 + 1} - \frac{R(x)}{\sqrt{H(x)}} \biggr) \, \dd x \biggr).
\end{equation*}
By Theorem~\ref{thm:mj_muE_intervals}, the exponents are given by
\begin{equation*}
m_1 = \frac{1}{\pi} \int_{b_1}^{b_2} \frac{-(x-z_1)}{\sqrt{\abs{H(x)}}} \, \dd 
x, \quad
m_2 = \frac{1}{\pi} \int_{b_3}^{b_4} \frac{x-z_1}{\sqrt{\abs{H(x)}}} \, \dd x
= 1 - m_1.
\end{equation*}
By Theorem~\ref{thm:sum_mj_aj_ell_intervals},
\begin{equation*}
\alpha = \frac{1}{2} (b_1 + b_2 + b_3 + b_4) - z_1,
\end{equation*}
and, by Theorem~\ref{thm:a1a2},
\begin{equation*}
\beta = \frac{\capacity(E)}{m_1^{m_1} m_2^{m_2}} \ee^{g_E(z_1)}
\quad \text{with} \quad
g_E(z_1) = \frac{1}{2} \int_{b_2}^{b_3} \frac{\abs{R(x)}}{\sqrt{H(x)}} \, \dd x,
\end{equation*}
and the centers of $L$ are given by
\begin{equation*}
a_1 = \alpha - m_2 \beta, \quad
a_2 = \alpha + m_1 \beta.
\end{equation*}
\end{remark}

%----- Example -----%

\begin{example}\label{ex:two_intervals}
Let $E = \cc{-1, -0.3} \cup \cc{0.1, 1}$.  We evaluate the formulas in Remark~\ref{rem:two_intervals} using Mathematica and Matlab and obtain (numbers truncated after $5$ decimal places)
\begin{align*}
z_1 &= -0.10209\ldots, & m_1 &= 0.46710\ldots, & m_2 &= 0.53289\ldots,\\
\alpha &= 0.00209\ldots, & g_E(z_1) &= 0.20383\ldots,& \capacity(E) &= 0.48978\ldots,\\
\beta &= 1.19846\ldots,& a_1 &= -0.63655\ldots,& a_2 &= 0.56190\ldots,
\end{align*}
so that all parameters of the set $L$ are obtained. The corresponding conformal map $\Phi : \comp{E} \to \comp{L}$ is obtained in Example~\ref{ex:two_intervals_continued}.
\end{example}

%----- Text -----%

Since the set $E$ in~\eqref{eqn:E_l_intervals} is symmetric with respect to
the real line, the centers $a_1, \ldots, a_\ell$ of $L$ and the critical points
$w_1, \ldots, w_{\ell-1}$ of $g_L$ (see the beginning of Section~\ref{sect:general_results}) are real and satisfy
\begin{equation} \label{eqn:interlacing_aj_wj}
a_1<w_1<a_2<w_2<a_3<\ldots<a_{\ell-1}<w_{\ell-1}<a_{\ell};
\end{equation}
see~\cite[Thm.~2.1\,(v)]{SchiefermayrSete-II}. By~\cite[Thm.~2.1\,(iv) and (v)]{SchiefermayrSete-II}, we have the correspondence $w_i = \Phi(z_i)$, with the orderings~\eqref{eqn:inequality_zi_bj} and~\eqref{eqn:interlacing_aj_wj}.

For $\ell = 3$ intervals, the centers $a_1, a_2, a_3$ can be computed by solving the non-linear system of equations~\eqref{eqn:sys_a1a2a3}, see Example~\ref{ex:three_intervals} below.

For $\ell \geq 4$, the critical points $w_1, \ldots, w_{\ell-1}$ of the
Green's function $g_L$ cannot be given in an adequate explicit form in terms of
$a_1, \ldots, a_\ell$, since the polynomial in~\eqref{eqn:poly_crit_pts_gL} has degree $\ell-1 \geq 3$.  Therefore, generalizing~\cite[Alg.~5.1]{SchiefermayrSete-II}, we derive an algorithm for numerically computing $a_1, \ldots, a_\ell$ and $w_1, \ldots, w_{\ell-1}$ for general $\ell$, where we use equations~\eqref{eqn:gL_gE_at_critical_pts} and~\eqref{eqn:sum_mj_aj_ell_intervals}.

%----- Algorithm -----%
\pagebreak
\begin{algorithm} \label{algo:aj}~ \\
\textbf{Input:} Endpoints $b_1, \ldots, b_{2 \ell}$ of $E = \cup_{j=1}^\ell 
\cc{b_{2j-1}, b_{2j}}$, exponents $m_1, \ldots, m_\ell$, critical points $z_1, \ldots, z_{\ell-1}$ of $g_E$,
absolute and relative tolerances $\operatorname{abstol}$ and $\operatorname{reltol}$. \\
\textbf{Output:} Centers $a_1, \ldots, a_\ell$ of $L$ and critical points $w_1, 
\ldots, w_{\ell-1}$ of $g_L$.

\noindent Initial values:
\begin{equation*}
\begin{aligned}
a_j^{[0]}&=\tfrac{1}{2}(b_{2j-1}+b_{2j}), \qquad j=1,\ldots,\ell \\
w_j^{[0]}&=\tfrac{1}{2}(b_{2j}+b_{2j+1}), \qquad j = 1, \ldots, \ell-1.
\end{aligned}
\end{equation*}
FOR $k = 0, 1, 2, \ldots$ DO
\begin{enumerate}
\renewcommand{\labelenumi}{\arabic{enumi}.}
\item Compute $a_1^{[k+1]} < \ldots < a_{\ell}^{[k+1]}$ such that
\begin{equation} \label{eqn:algorithm}
\begin{aligned}
g_L(w_i^{[k]}) &= g_E(z_i), \quad i = 1, \ldots, \ell-1,\\
\sum_{j=1}^{\ell} m_j a_j^{[k+1]} &= \frac{1}{2} \sum_{j=1}^{2\ell} b_j - 
\sum_{j=1}^{\ell-1} z_j,
\end{aligned}
\end{equation}

\item Compute the solutions $w_1^{[k+1]} < \ldots < w_{\ell-1}^{[k+1]}$ 
of the equation
\begin{equation*}
\sum_{j=1}^\ell m_j \prod_{\substack{i=1 \\ i \neq j}}^\ell \Bigl( 
w-a_i^{[k+1]} \Bigr) = 0.
\end{equation*}

\item Stop if $\bigl\lvert a_j^{[k+1]} - a_j^{[k]} \bigl\lvert < 
\operatorname{abstol} + \operatorname{reltol} \cdot \bigl\lvert a_j^{[k]} 
\bigl\lvert$ for all $j = 1, \ldots, \ell$.
\end{enumerate}
ENDFOR
\end{algorithm}

%----- Remark -----%

\begin{remark}
\begin{enumerate}
\item When solving the (non-linear) system of equations~\eqref{eqn:algorithm} with an iterative method like Newton's method, one can use $a_1^{[k]}, \ldots, a_\ell^{[k]}$ as initial values.
\item In all our examples, we used $\operatorname{abstol} = \operatorname{reltol} = 10^{-13}$ in Algorithm~\ref{algo:aj}.
\item Algorithm~\ref{algo:aj} extends to sets $E = \cup_{j=1}^\ell E_j$ consisting of $\ell$ components that are each symmetric with respect to the real line, i.e., with $E_j^* = E_j$ for $j = 1, \ldots, \ell$. In that case, $E_j \cap \R$ is an interval or a point and we define $b_1,\ldots, b_{2 \ell}$ by $E_j \cap \R = \cc{b_{2j-1}, b_{2j}}$; compare~\cite[Rem.~5.2]{SchiefermayrSete-II}. Then one can run the above algorithm, provided that one can obtain the exponents $m_1, \ldots, m_\ell$, the coefficient $\alpha$ in Theorem~\ref{thm:sum_mj_aj}, and the values $g_E(z_j)$ of the Green's function of $\comp{E}$ at its critical points.
\end{enumerate}
\end{remark}

\begin{example} \label{ex:analytic}
We consider the following three families of sets $E$ consisting of real 
intervals for which the parameters of the lemniscatic domain corresponding to 
$\comp{E}$ are known explicitly.
\begin{enumerate}
\item Two symmetric intervals: $E = [-b_4, -b_3] \cup [b_3, b_4]$ with $0 < b_3 
< b_4$ for which $m_1 = m_2 = 1/2$ and $a_2 = (b_3 + b_4)/2$, $a_1 = -a_2$, and 
$\capacity(E) = \sqrt{b_4^2 - b_3^2} / 2$; see~\cite[Cor.~3.3]{SeteLiesen2016} 
or~\cite[Ex.~3.5]{SchiefermayrSete-II}.

\item Two intervals: $E = \cc{-1, b_2} \cup \cc{b_3, 1}$ with
$b_2 = \frac{1}{2} (1 - \alpha^2) - \alpha$,
$b_3 = \frac{1}{2} (1 - \alpha^2) + \alpha$
for $0 < \alpha < 1$; see \cite[Ex.~3.2]{SchiefermayrSete-II}.

\item Three symmetric intervals: $E = \cc{-1, -(1 - \alpha)} \cup 
\cc{-\alpha, \alpha} \cup \cc{1-\alpha, 1}$
with $0 < \alpha < 1/2$; see~\cite[Ex.~4.3]{SchiefermayrSete-II}.
\end{enumerate}
We applied Algorithm~\ref{algo:aj} to these three examples.  
Table~\ref{tab:analytic_examples} displays the number of iteration steps until 
convergence, the maximal error $\max_{j=1, \ldots, \ell} \abs{a_j^{[k]} - 
a_j}$ in the final step of the iteration, as well as the maximal error for 
$m_1, \ldots, m_\ell$.
We observe that the algorithm returns very accurate 
approximations of $a_1, \ldots, a_\ell$.
Figure~\ref{fig:analytic_examples} shows the error curves $\abs{a_j^{[k]} - 
a_j}$ in examples (ii) and (iii),
which underlines the very fast convergence of Algorithm~\ref{algo:aj},
which seems to be quadratic in these examples.
Only $4$ iteration steps are needed until convergence, and the
same observations holds for the examples in Section~\ref{sect:map}
with $3, 4, 8$ or $10$ intervals, where the prescribed tolerance is
achieved in at most~$5$ iteration steps.

In all three examples, the set $E$ is a polynomial pre-image of $\cc{-1, 1}$,
i.e., has the form $E = P^{-1}(\cc{-1, 1})$ where $P$ is a polynomial.
Note that we ran Algorithm~\ref{algo:aj} only with the endpoints $b_1, \ldots, 
b_{2 \ell}$ as inputs, without any information on $P$.
In contrast, the earlier method in~\cite{SchiefermayrSete-II} is limited
to polynomial pre-images and requires knowledge of the polynomial $P$.
\end{example}

\begin{table}[t]
{\centering
\begin{tabular}{lccc}
\toprule
set $E$ & iter. steps & max. error $\abs{a_j^{[k]} - a_j}$ & max. error for 
$m_j$ \\
\midrule
(i) for $b_3 = 1$, $b_4 = 2$ & $1$ & $3.5660 \cdot 10^{-13}$ & $1.9479 \cdot 
10^{-13}$\\
(ii) with $\alpha = 0.05$ & $4$ & $3.4994 \cdot 10^{-13}$ & $5.9341 \cdot 
10^{-14}$ \\
(iii) with $\alpha = 0.4$ & $4$ & $2.7023 \cdot 10^{-13}$ & $4.3743 \cdot 
10^{-14}$ \\
\bottomrule
\end{tabular}

}
\caption{Applying Algorithm~\ref{algo:aj} to the sets in 
Example~\ref{ex:analytic}: number of iteration steps until convergence, 
and final maximal error between the computed and exact values of $a_1, \ldots, 
a_\ell$ and $m_1, \ldots, m_\ell$.}
\label{tab:analytic_examples}
\end{table}

\begin{figure}
{\centering
\includegraphics[width=0.48\linewidth]{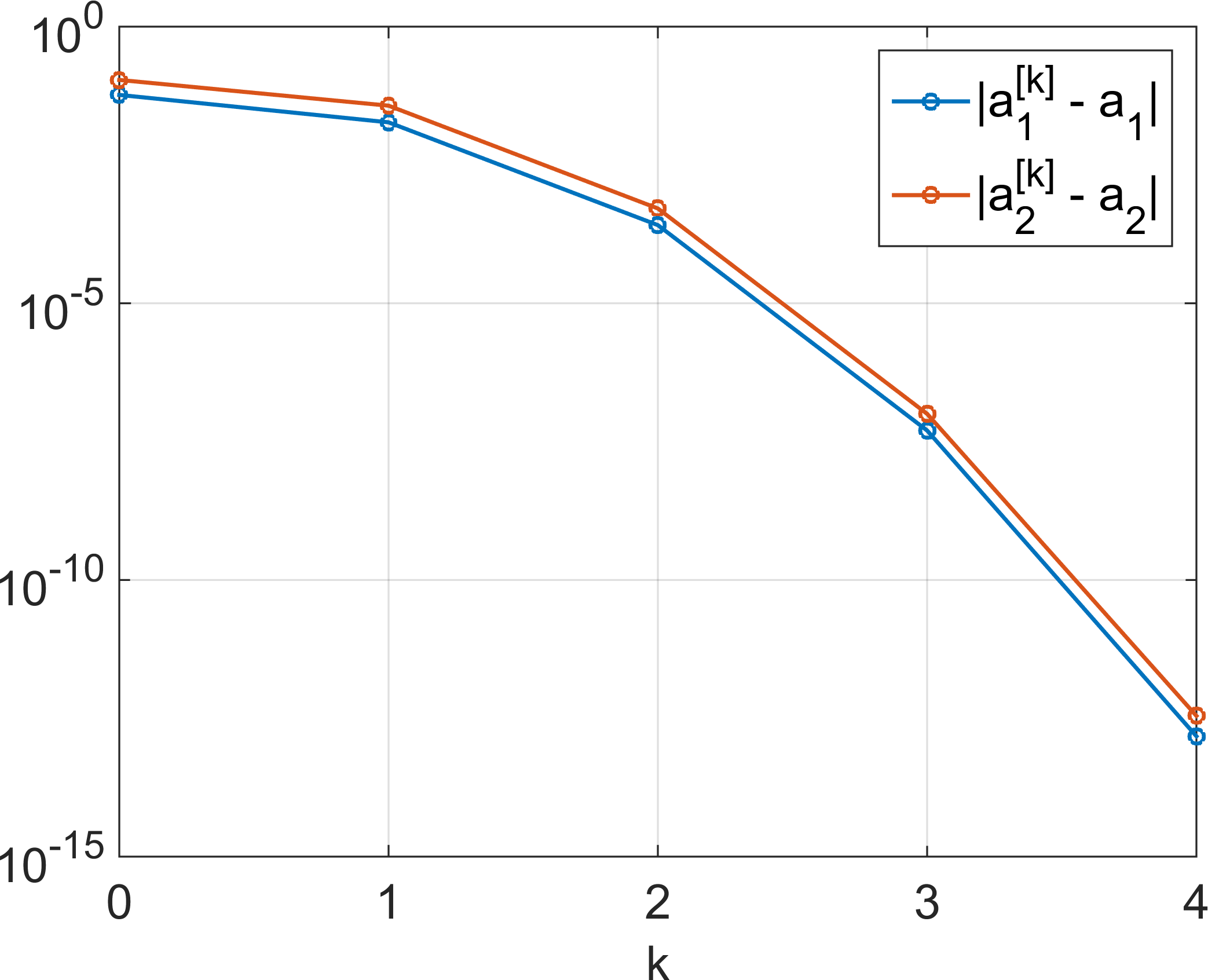}
\includegraphics[width=0.48\linewidth]{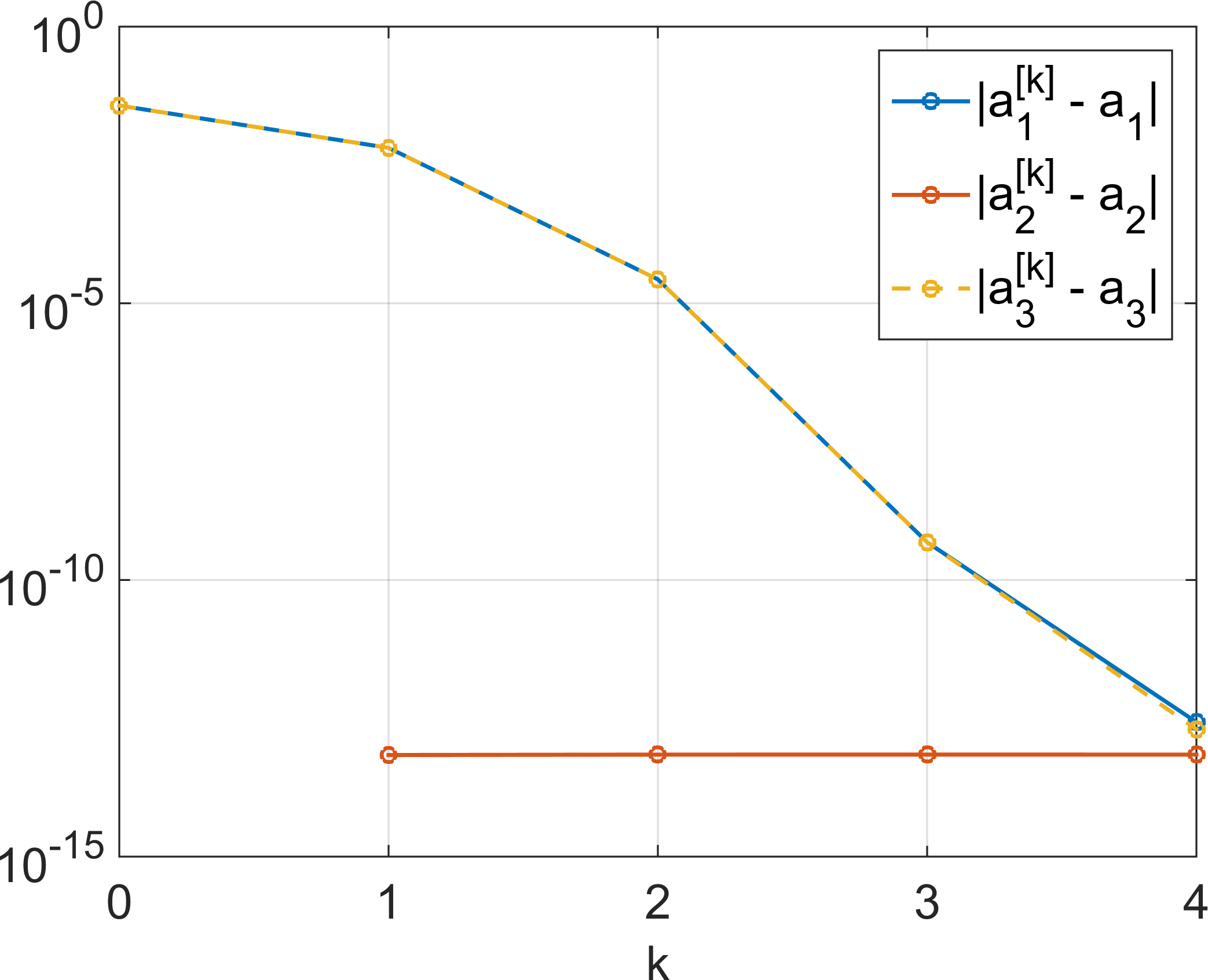}

}
\caption{Error curves $\abs{a_j^{[k]} - a_j}$ in Algorithm~\ref{algo:aj}
applied to the sets $E$ in Example~\ref{ex:analytic} (ii) (left) and 
(iii) (right).
Missing dots mean that the error is exactly zero.}
\label{fig:analytic_examples}
\end{figure}

\begin{example} \label{ex:rand_intervals}
As a numerical experiment, we considered $500$ examples 
with random endpoints generating $5$ and $10$ intervals, respectively.
In most cases, Algorithm~\ref{algo:aj} converged in $4$ or $5$ steps,
and in all cases in at most $7$ steps.
\end{example}

\pagebreak
%----------------------------------------------------------%
\section{The Conformal Map for Several Intervals}\label{sect:map}
%----------------------------------------------------------%

Let $E$ be the union of $\ell$ disjoint real intervals given 
in~\eqref{eqn:E_l_intervals}. The logarithmic capacity $\capacity(E)$,
the exponents $m_1, \ldots, m_\ell$ and the centers $a_1, \ldots, a_\ell$ 
of the canonical domain $\comp{L}$ have been obtained in Section~\ref{sect:domain_for_intervals}.
Next, we address the
computation of the conformal map $\Phi : \comp{E} \to \comp{L}$.

For the computation of $\Phi(z)$ for real $z$, we need the following notation.
The set $\R \setminus E$ consists of $\ell+1$ open intervals,
denoted by $\R \setminus E = I_0 \cup \ldots \cup I_\ell$, see~\eqref{eqn:gaps}.
By~\cite[Thm.~2.1\,(iv)]{SchiefermayrSete-II}, the set $\R \setminus L$ consists of $\ell+1$ open intervals
\begin{equation} \label{eqn:intervals_Jk}
J_0 = \oo{-\infty, c_1}, \quad
J_k = \oo{c_{2k}, c_{2k+1}} \quad \text{for } k = 1, \ldots, \ell-1, \quad
J_\ell = \oo{c_{2 \ell}, \infty},
\end{equation}
where
\begin{equation} \label{eqn:Phi_Ik_Jk}
\Phi(I_k) = J_k, \quad k = 0, \ldots, \ell,
\end{equation}
and $\Phi$ is strictly increasing
on~$\R \setminus E$.  In particular, $\Phi(b_j) = c_j \in \partial L$, $j = 1, \ldots, 2 \ell$, and
\begin{equation} \label{eqn:inequality_cj_aj}
-\infty < c_1 < a_1 < c_2 < c_3 < a_2 < c_4 < c_5 < a_3 < \ldots <
c_{2 \ell} < + \infty.
\end{equation}

Our first result extends the identity $g_E(z) = g_L(\Phi(z))$, see~\eqref{eqn:gE_gL}, to the \emph{complex} Green's functions.

\begin{theorem} \label{thm:Phi_eqn}
Let $E$ be as in~\eqref{eqn:E_l_intervals} and $I_0, \ldots, I_\ell$ be as in~\eqref{eqn:gaps}, then the following hold.
\begin{enumerate}
\item \label{it:Phi_complex}
For $z \in \C \setminus \oc{-\infty, b_{2 \ell}}$,
\begin{equation} \label{eqn:Phi_complex}
\sum_{j=1}^\ell m_j \log(\Phi(z) - a_j) - \log(\capacity(E))
= \int_{b_{2 \ell}}^z \frac{R(\zeta)}{\sqrt{H(\zeta)}} \, \dd \zeta
\end{equation}
with the principal branch of the logarithm and an integration path from
$b_{2\ell}$ to $z$ that lies in $\C \setminus \oc{-\infty, b_{2 \ell}}$
except for its starting point.

\item \label{it:Phi_real_z}
For $z \in I_k$, $k = 0, \ldots, \ell$,
\begin{equation} \label{eqn:Phi_real}
\sum_{j=1}^\ell m_j \log \abs{\Phi(z)-a_j} - \log(\capacity(E))
= \int_b^z \frac{R(x)}{\sqrt{H(x)}} \, \dd x,
\end{equation}
where $\log$ is the real logarithm, the integration is along the real line and 
$b \in \{ b_1, \ldots, b_{2 \ell} \}$ is an endpoint of $I_k$.
\end{enumerate}
\end{theorem}

\begin{proof}
By~\eqref{eqn:gE_gL}, we have $g_E(z) = g_L(\Phi(z))$, where $g_L$ is given 
by~\eqref{eqn:gL}.  By~\eqref{eqn:wirtinger_gE_gL},
\begin{equation} \label{eqn:Green_derivative}
2 \partial_z g_E(z) = \sum_{j=1}^\ell m_j \frac{\Phi'(z)}{\Phi(z) - a_j}, \quad 
z \in \C \setminus E,
\end{equation}
and by Theorem~\ref{thm:gE_intervals}\,\ref{it:green_L_intervals} 
and~\eqref{eqn:wirtinger}, $2 \partial_z g_E(z) = R(z)/\sqrt{H(z)}$.

\ref{it:Phi_complex}
For $z \in \C \setminus \oc{-\infty, b_{2 \ell}}$, 
we integrate~\eqref{eqn:Green_derivative} along a path from $b_{2 \ell}$ to $z$ that lies in the simply 
connected domain $\C \setminus \oc{-\infty, b_{2 \ell}}$ (except for the
initial point $b_{2 \ell}$) and obtain
\begin{align*}
\int_{b_{2 \ell}}^z \frac{R(\zeta)}{\sqrt{H(\zeta)}} \, \dd \zeta
&= \int_{b_{2 \ell}}^z 2 \partial_z g_E(\zeta) \, \dd \zeta
= \sum_{j=1}^\ell m_j \int_{b_{2 \ell}}^z \frac{\Phi'(\zeta)}{\Phi(\zeta) - 
a_j} \, \dd \zeta \\
&= \sum_{j=1}^\ell m_j \log(\Phi(z) - a_j)
- \sum_{j=1}^\ell m_j \log(\Phi(b_{2\ell}) - a_j)
\end{align*}
with the principal branch of the logarithm.
Since $\Phi(b_{2 \ell}) \in \partial L$, we have $g_L(\Phi(b_{2 \ell})) = 0$.
Together with $\Phi(b_{2 \ell}) > a_j$ for $j = 1, \ldots, \ell$ 
and~\eqref{eqn:gL}, this implies
\begin{equation} \label{eqn:Phi_at_endpoint}
\sum_{j=1}^\ell m_j \log(\Phi(b_{2\ell}) - a_j)
= \sum_{j=1}^\ell m_j \log \abs{\Phi(b_{2\ell}) - a_j} = \log(\capacity(E)).
\end{equation}
This completes the proof of~\ref{it:Phi_complex}.

\ref{it:Phi_real_z}
Let $k \in \{ 0, \ldots, \ell \}$.
For $z \in I_k$, we integrate~\eqref{eqn:Green_derivative} from
one endpoint $b \in \{ b_1, \ldots, b_{2k} \}$ of $I_k$ to~$z$ along the real
line.
Similarly to~\ref{it:Phi_complex}, we have
\begin{equation*}
\int_b^z \frac{R(x)}{\sqrt{H(x)}} \, \dd x
= \sum_{j=1}^\ell m_j \int_b^z \frac{\Phi'(x)}{\Phi(x) - a_j} \, \dd x
= \sum_{j=1}^\ell m_j \log \abs{\Phi(z) - a_j} - \log(\capacity(E)).
\end{equation*}
Note that $\Phi(x) \neq a_j$ for all $x \in I_k$,
which follows from~\eqref{eqn:Phi_Ik_Jk} and~\eqref{eqn:inequality_cj_aj}.
This completes the proof of~\ref{it:Phi_real_z}.
\end{proof}

\begin{remark}
If we replace $b_{2 \ell}$ by $b \in \{ b_{2 k}, b_{2 k + 1} \}$ for some
$k \in \{ 0, \ldots, \ell-1 \}$ (where $b_0 \coloneq b_1$)
in Theorem~\ref{thm:Phi_eqn}\,\ref{it:Phi_complex} then we obtain similarly 
to~\eqref{eqn:Phi_at_endpoint}
\begin{equation*}
\sum_{j=1}^\ell m_j \log(\Phi(b) - a_j)
= \log(\capacity(E)) \pm \ii \pi (m_{k+1} + \ldots + m_\ell),
\end{equation*}
since then $\Phi(b) < a_j$ for $j = k+1, \ldots, \ell$ and $\Phi(b) > a_j$ for 
$j = 1, \ldots, k$.  Therefore, for $z \in \C \setminus \R$, $\Phi(z)$ satisfies
\begin{equation*}
\sum_{j=1}^\ell m_j \log(\Phi(z) - a_j) - \log(\capacity(E)) \mp \ii \pi 
(m_{k+1} + \ldots + m_\ell)
= \int_b^z \frac{R(\zeta)}{\sqrt{H(\zeta)}} \, \dd \zeta,
\end{equation*}
where the sign of ``$\mp$'' is ``$-$'' if $\im(z) > 0$ and ``$+$'' if $\im(z) < 
0$.
This shows that, for $b \in \{ b_{2 k}, b_{2 k + 1} \}$, the complex Green's 
function of $\comp{E}$ satisfies
\begin{equation}
\int_b^z \frac{R(\zeta)}{\sqrt{H(\zeta)}} \, \dd \zeta
= \int_{b_{2 \ell}}^z \frac{R(\zeta)}{\sqrt{H(\zeta)}} \, \dd \zeta \mp \ii \pi 
(m_{k+1} + \ldots + m_\ell), \quad z \in \C \setminus \R.
\end{equation}
This shows that for the real Green's function of $\comp{E}$, see
Theorem~\ref{thm:gE_intervals}\,\ref{it:green_L_intervals},
the point $b$ can indeed be any of the endpoints $\{ b_1, \ldots, b_{2 \ell} \}$.
\end{remark}

Equations~\eqref{eqn:Phi_complex} and~\eqref{eqn:Phi_real} in
Theorem~\ref{thm:Phi_eqn} can be solved numerically in order to obtain 
the value $\Phi(z)$.
More precisely, if $z \in \C \setminus \R$, we solve
\begin{equation} \label{eqn:Phi_eqn_complex}
F_1(w) \coloneq \sum_{j=1}^\ell m_j \log(w - a_j) - \log(\capacity(E))
= \int_{b_{2\ell}}^z \frac{R(\zeta)}{\sqrt{H(\zeta)}} \, \dd \zeta,
\quad w \in \C \setminus \R,
\end{equation}
with the principal branch of the logarithm, and, if $z \in I_k$, $k = 0, 
\ldots, \ell$, we solve
\begin{equation} \label{eqn:Phi_eqn_real}
F_2(w) \coloneq \sum_{j=1}^\ell m_j \log \abs{w - a_j} - \log(\capacity(E))
= \int_b^z \frac{R(x)}{\sqrt{H(x)}} \, \dd x, \quad w \in \R \setminus L,
\end{equation}
with real logarithms, and where $b \in \{ b_1, \ldots, b_{2 \ell} \}$ is an 
endpoint of $I_k$.  
By Theorem~\ref{thm:Phi_eqn}, $w = \Phi(z)$ is one solution, and we discuss its 
uniqueness next.

\begin{theorem} \label{thm:Phi_eqn_unique_sol}
Let $E$ be as in~\eqref{eqn:E_l_intervals} and $I_0, \ldots, I_\ell$ and $J_0, \ldots, J_\ell$ be as in~\eqref{eqn:gaps} and~\eqref{eqn:intervals_Jk}, respectively, then the following holds.
\begin{enumerate}
\item \label{it:Phi_eqn_unique_complex}
For $z \in \C \setminus \R$, equation~\eqref{eqn:Phi_eqn_complex} has the 
unique solution $w = \Phi(z)$.

\item \label{it:Phi_eqn_unique_real}
For $z \in \R \setminus E$, equation~\eqref{eqn:Phi_eqn_real} has the unique 
solution $w = \Phi(z)$ for $w$ restricted to a suitable interval depending on $z$:
\begin{enumerate}
\renewcommand{\labelenumii}{\textup{(\alph{enumii})}}
\item If $z \in I_0$ then $w \in J_0$.
\item If $z \in I_\ell$ then $w \in J_\ell$.
\item If $z \in I_k$, $k = 1, \ldots, \ell-1$, then
\begin{itemize}
\item if $z \in \oo{b_{2k}, z_k}$ then $w \in \oo{c_{2k}, w_k}$,
\item if $z = z_k$ then $w = w_k$,
\item if $z \in \oo{z_k, b_{2k+1}}$ then $w \in \oo{w_k, c_{2k+1}}$,
\end{itemize}
where $z_1, \ldots, z_{\ell-1}$ are the critical points of $g_E$,  compare 
Corollary~\ref{cor:intervals}\,\ref{it:crit_pts_green}, and 
$w_1, \ldots, w_{\ell-1}$ are the critical points of $g_L$, 
compare~\eqref{eqn:interlacing_aj_wj}.
\end{enumerate}
\end{enumerate}
\end{theorem}

\begin{proof}
\ref{it:Phi_eqn_unique_complex}
In order to show that~\eqref{eqn:Phi_eqn_complex} has a unique solution $w$,
we consider the mapping properties of the function $F_1$ in the 
upper half-plane $\bH = \{ w \in \C : \im(w) > 0 \}$.
Since $F_1'(w) = \sum_{j=1}^\ell m_j / (w-a_j)$, the critical points of $F_1$ are the critical points $w_1, \ldots, w_{\ell-1}$ of $g_L$, compare~\eqref{eqn:gL_wirtinger}, which are real by~\eqref{eqn:interlacing_aj_wj}, hence $F_1$ is conformal in $\bH$.
The function $F_1$ can be written as
\begin{equation*}
F_1(w) = \sum_{j=1}^\ell m_j \log \abs{w - a_j} - \log(\capacity(E))
+ \ii \sum_{j=1}^\ell m_j \arg(w - a_j) = g_L(w) + \ii h_L(w),
\end{equation*}
with the principal branch of the argument.
Thus $\arg(w - a_j) \in \oo{0, \pi}$ for $\im(w) > 0$ and $F_1$ maps $\bH$
into the strip $\{ \zeta \in \C : 0 < \im(\zeta) < \pi \}$.
Let us consider the behaviour of $F_1$ on the real line.  
The function $h_L$ has the values
\begin{equation*}
h_L(w) =
\begin{cases}
0, & w > a_\ell, \\
(m_{k+1} + \ldots + m_\ell) \pi, & a_k < w < a_{k+1}, \\
\pi, & w < a_1,
\end{cases}
\end{equation*}
and $g_L$ satisfies $\lim_{w \to \pm \infty} g_L(w) = + \infty$ and
$\lim_{w \to a_j} g_L(w) = - \infty$.
Moreover, $g_L$ has the critical points $w_1, \ldots, w_{\ell-1}$ which
satisfy~\eqref{eqn:interlacing_aj_wj} and
$g_L(w_k) > 0$ for $k = 1, \ldots, \ell-1$.
Therefore,
\begin{align*}
F_1(\oo{a_\ell, +\infty}) &= \oo{-\infty, + \infty}, \\
F_1(\oc{a_k, w_k}) &= F_1(\co{w_k, a_{k+1}})
= \oc{-\infty, g_L(w_k)} + \ii (m_{k+1} + \ldots + m_\ell) \pi, \\
F_1(\oo{-\infty, a_1}) &= \oo{-\infty, +\infty} + \ii \pi.
\end{align*}
This shows that $F_1$ maps the upper half-plane $\bH$ conformally into the strip 
with slits
\begin{equation*}
S = \{ \zeta \in \C : 0 < \im(\zeta) < \pi \} \setminus 
\bigcup_{k=1}^{\ell-1} \Bigl( \oc{-\infty, g_L(w_k)} + \ii (m_{k+1} + \ldots + 
m_\ell) \pi \Bigr).
\end{equation*}
By distinguishing the two sides of the slits, $F_1$ maps $\partial \bH$ bijectively onto the boundary of $S$ and thus is univalent in $\bH$.
Since $F_1(w) = \conj{F_1(\conj{w})}$ for all $w \in \C \setminus \R$, the 
function $F_1$ is univalent in $\C \setminus \R$.
Together with Theorem~\ref{thm:Phi_eqn}, this shows that $w = \Phi(z)$
is the unique solution of~\eqref{eqn:Phi_eqn_complex}.

\begin{figure}
{\centering
\includegraphics[width=0.48\linewidth]{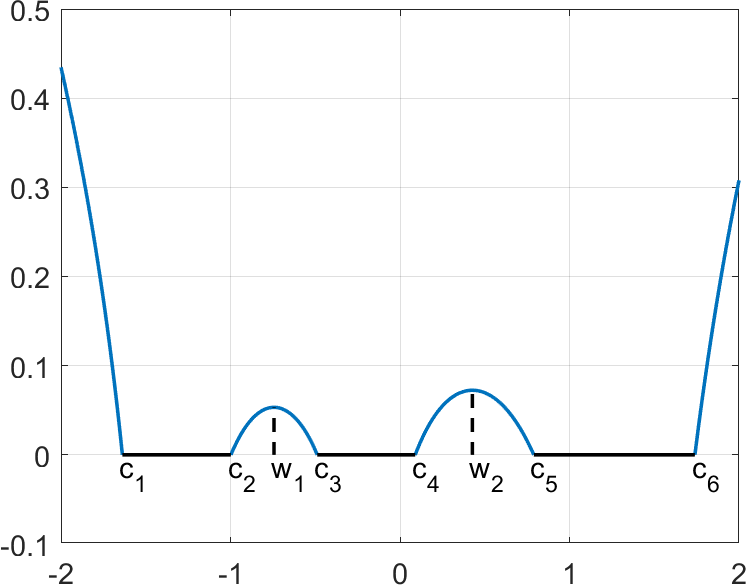}

}
\caption{Illustration of the proof of 
Theorem~\ref{thm:Phi_eqn_unique_sol}\,\ref{it:Phi_eqn_unique_real}:  The graph of the function $F_2(w) = g_L(w)$ for $w \in \R \setminus L$.}
\label{fig:gL_real}
\end{figure}

\ref{it:Phi_eqn_unique_real}
The proof relies on the piece-wise monotonicity of $F_2$ in $\R \setminus L$.
Fix $k \in \{ 0, \ldots, \ell \}$ and $z \in I_k$.  Then $\Phi(z) \in J_k$, 
see~\eqref{eqn:Phi_Ik_Jk}, and we solve~\eqref{eqn:Phi_eqn_real} for $w \in J_k$.
The function $F_2 : J_k \to \R$ is real analytic with derivative
\begin{equation*}
F_2'(w) = \sum_{j=1}^\ell \frac{m_j}{w - a_j}.
\end{equation*}
For $k = 1, \ldots, \ell-1$, the function $F_2$ satisfies
$F_2(w) > 0$ for $w \in J_k = \oo{c_{2k}, c_{2k+1}}$ and $F_2(c_{2k}) = 
F_2(c_{2k+1}) = 0$, and $F_2$ has a unique critical point $w_k$ with $c_{2k} < 
w_k < c_{2k+1}$.
Therefore, $F_2$ is strictly increasing on $\cc{c_{2k}, w_k}$ and strictly 
decreasing on $\cc{w_k, c_{2k+1}}$;
see Figure~\ref{fig:gL_real} for an illustration.
Since $\Phi(z_k) = w_k$, where $z_k$ is the critical point of $g_E$ in $I_k$,
and taking into account the mapping properties of $\Phi$ 
in~\cite[Thm.~2.1\,(v)]{SchiefermayrSete-II}, this leads to the following 
statement.
If $b_{2k} < z < z_k$ then~\eqref{eqn:Phi_eqn_real} has the unique solution $w 
= \Phi(z)$ with $c_{2k} < w < w_k$.
If $z_k < z < b_{2k+1}$ then~\eqref{eqn:Phi_eqn_real} has the unique solution 
$w = \Phi(z)$ with $w_k < w < c_{2k+1}$.
Since $F_2'(w) < 0$ for $w \in J_0$ and $F_2'(w) > 0$ for $w \in J_\ell$, 
in both cases, equation~\eqref{eqn:Phi_eqn_real} has the unique solution $w = 
\Phi(z)$.
This completes the proof of~\ref{it:Phi_eqn_unique_real}.
\end{proof}

\begin{remark}
The points $c_1, \ldots, c_{2 \ell}$ in~\eqref{eqn:intervals_Jk}
are the real zeros of the Green's function~$g_L$ in~\eqref{eqn:gL} and can
be computed, e.g., with bisection or Newton's method.
\end{remark}

Next, we discuss the numerical solution of
equations~\eqref{eqn:Phi_eqn_complex} and~\eqref{eqn:Phi_eqn_real}.
The derivatives of both functions $F_1$ and $F_2$ are
$\sum_{j=1}^\ell m_j/(w - a_j)$.
We solve~\eqref{eqn:Phi_eqn_complex} and~\eqref{eqn:Phi_eqn_real} with the 
damped Newton iteration
\begin{equation}
w^{[k+1]} = w^{[k]} - d \frac{F_i(w^{[k]}) - \int_b^z 
\frac{R(\zeta)}{\sqrt{H(\zeta)}} \, \dd \zeta}{\sum_{j=1}^\ell 
\frac{m_j}{w^{[k]} - a_j}},
\quad k = 0, 1, 2, \ldots,
\end{equation}
where we use the damping parameter $d = 1, 2^{-1}, 2^{-2}, \ldots, 2^{-10}$.
In view of Theorem~\ref{thm:Phi_eqn_unique_sol}, we choose the initial point
$w^{[0]}$ for the individual cases as follows.
For given $z \in \C \setminus \R$, we solve~\eqref{eqn:Phi_eqn_complex}
with $w^{[0]} = z$.
For $z \in \R \setminus E$, we solve~\eqref{eqn:Phi_eqn_real} with
\begin{equation*}
w^{[0]} =
\begin{cases}
z & \text{if } z \in I_0 \text{ or } z \in I_\ell, \\
c_{2k} + \frac{z-b_{2k}}{z_k-b_{2k}} (w_k-c_{2k}) & \text{if } z \in \oo{b_{2k}, z_k}, \\
w_k + \frac{z-z_k}{b_{2k+1}-z_k} (c_{2k+1}-w_k) & \text{if } z \in \oo{z_k, b_{2k+1}}.
\end{cases}
\end{equation*}
Note that for $z \in \oo{b_{2k}, z_k}$ or $z \in \oo{z_k, b_{2k+1}}$ one could 
also use only bisection to solve~\eqref{eqn:Phi_eqn_real}.  However, the Newton iteration typically converges faster than bisection.

\begin{example} \label{ex:two_intervals_continued}
Let $E = \cc{-1, -0.3} \cup \cc{0.1, 1}$ be as in 
Example~\ref{ex:two_intervals},
where we already computed all parameters of the set $L$.
The values $\Phi(z)$ of the conformal map for $z \in \C \setminus E$ are 
obtained as described above.
Figure~\ref{fig:two_intervals} shows a grid and its image under $\Phi$ as
well as the set $E$ and the corresponding set $L$.
\end{example}

\begin{figure}
{\centering
\includegraphics[width=0.48\linewidth]{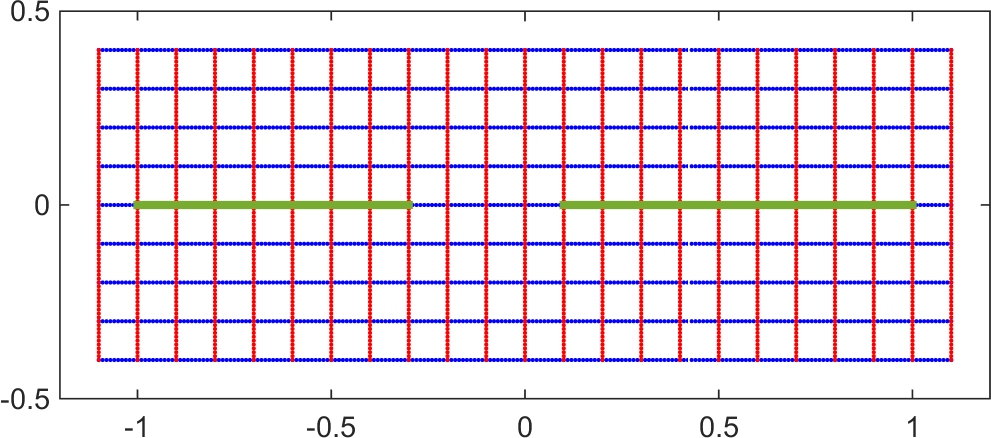}
\includegraphics[width=0.48\linewidth]{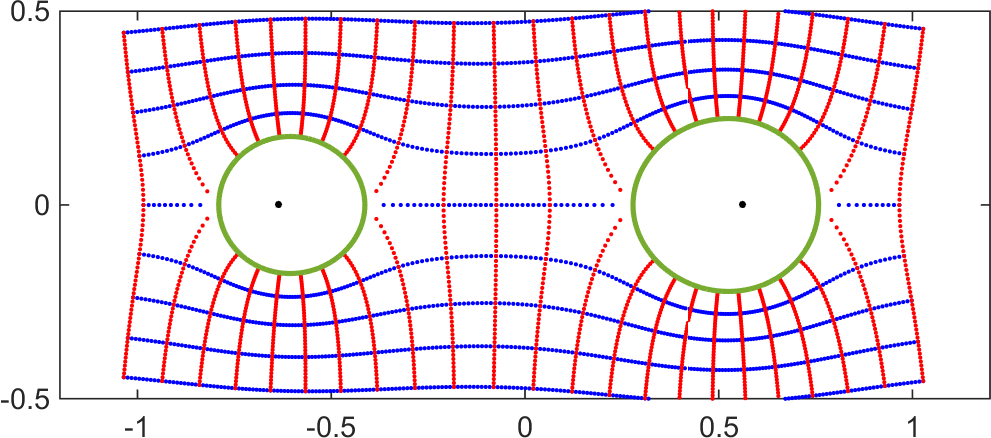}

}
\caption{Left: Set $E = \cc{-1, -0.3} \cup \cc{0.1, 1}$ in 
Example~\ref{ex:two_intervals_continued} with a grid.
Right: $\partial L$ (green curves), $a_1, a_2$ (black dots) and the image of 
the grid under $\Phi$.}
\label{fig:two_intervals}
\end{figure}

\begin{example} \label{ex:three_intervals}
Let $E = \cc{-2, -0.9} \cup \cc{-0.7, 0.2} \cup \cc{0.5, 2.2}$, i.e., $\ell = 3$.
Let us compute the parameters of the corresponding lemniscatic domain~$\comp{L}$.
First, we compute the coefficients of the polynomial~$R$ with the help of~\eqref{eqn:coeff_R}.
Using Theorem~\ref{thm:mj_muE_intervals}, we obtain the exponents 
$m_1 = 0.3601$, $m_2 = 0.1772$, $m_3 = 0.4627$ (all values rounded to four decimal places).
The capacity $\capacity(E) = 1.0458$ is derived from Corollary~\ref{cor:intervals}\,\ref{it:cap_L_intervals}.

Next, we calculate $a_1, a_2, a_3$ in two ways:
first, by solving the non-linear system of equations~\eqref{eqn:sys_a1a2a3},
and second, using Algorithm~\ref{algo:aj}.  In both variants, we obtain
$a_1 = -1.4101$, $a_2 = -0.1950$, $a_3 = 1.3896$,
where the values agree up to seven digits.
Algorithm~\ref{algo:aj} converges in $4$~steps.
The set $L$ is shown in Figure~\ref{fig:three_intervals} (right), which also 
shows the values of $\Phi$ on a grid.
\end{example}

\begin{figure}
{\centering
\includegraphics[width=0.48\linewidth]{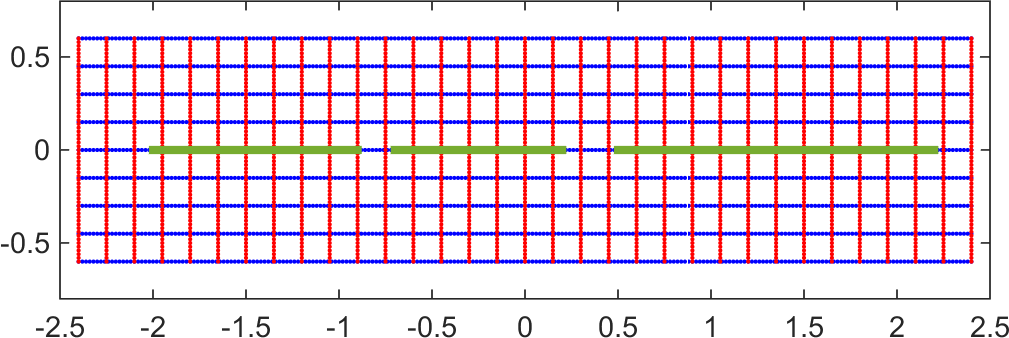}
\includegraphics[width=0.48\linewidth]{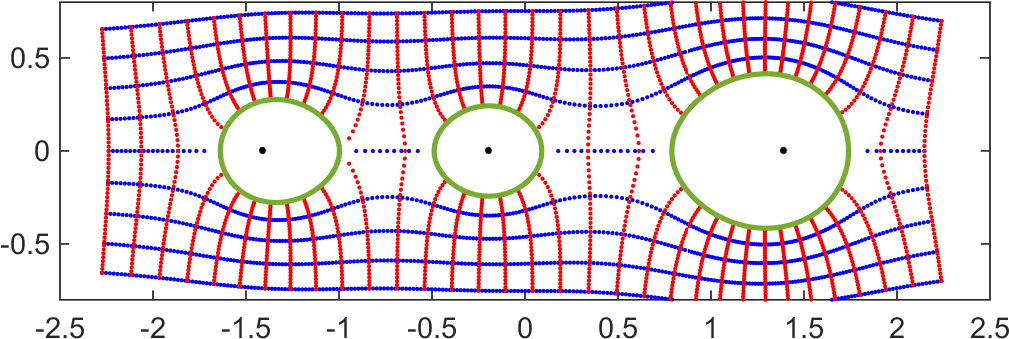}

}
\caption{Left: Set $E = \cc{-2, -0.9} \cup \cc{-0.7, 0.2} \cup \cc{0.5, 2.2}$ 
in Example~\ref{ex:three_intervals} with a grid.
Right: $\partial L$ (green curves), $a_1, a_2, a_3$ (black dots) and the image 
of the grid under $\Phi$.}
\label{fig:three_intervals}
\end{figure}

\begin{figure}[t!]
{\centering
\includegraphics[width=0.98\linewidth]{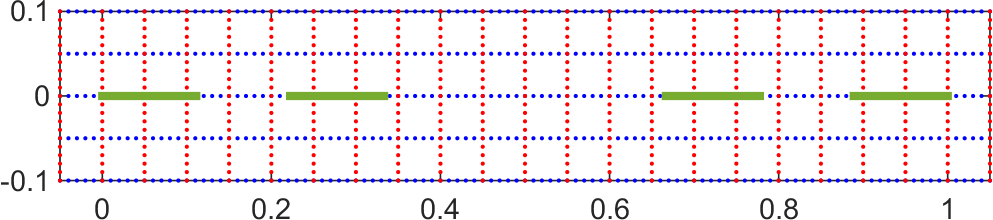}

\includegraphics[width=0.98\linewidth]{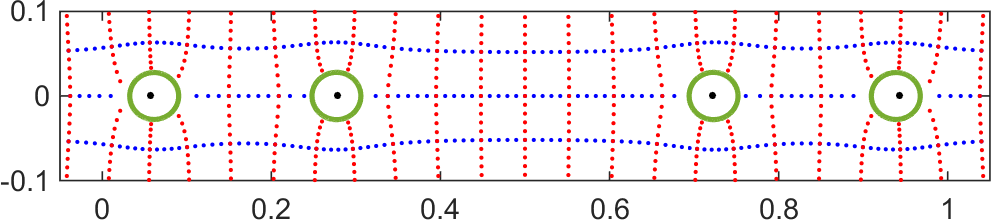}

\includegraphics[width=0.98\linewidth]{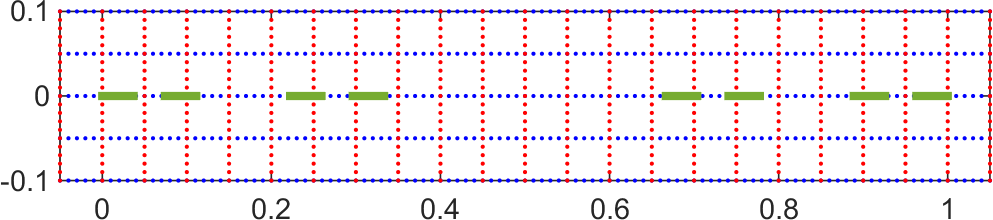}

\includegraphics[width=0.98\linewidth]{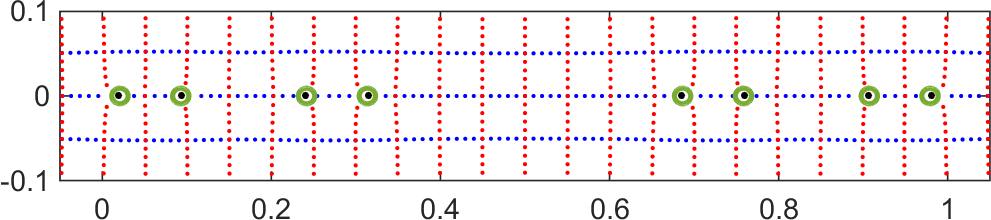}

}
\caption{Panels 1 and 3:
Sets $E_2$ and $E_3$ from the construction of the Cantor middle third 
set (see Example~\ref{ex:cantor}) with a grid.
Panels 2 and 4: Corresponding sets $L$ (with $\partial L$ in green), centers 
$a_j$ (black dots), and image of the grid under $\Phi$.}
\label{fig:cantor}
\end{figure}

\begin{example} \label{ex:cantor}
The classical \emph{Cantor middle third set} is defined by
\begin{equation*}
E = \bigcap_{k=0}^\infty E_k, \quad \text{where }
E_0 = \cc{0, 1}, \quad
E_{k+1} = \frac{1}{3} E_k \cup \left( \frac{2}{3} + \frac{1}{3} E_k \right), 
\quad k \geq 0.
\end{equation*}
We consider the two sets $E_2$ and $E_3$ which consist of $4$ and $8$ 
intervals, respectively.  The logarithmic capacities, numerically computed 
with~\eqref{eqn:cap_L_intervals_1}, are
\begin{equation*}
\capacity(E_2) = 0.228430704425168 \quad \text{and} \quad
\capacity(E_3) = 0.224752818755217
\end{equation*}
and agree in the first 12 digits with the values found 
in~\cite[Ex.~4.13]{LiesenSeteNasser2017}.
We compute the centers with Algorithm~\ref{algo:aj}, which numerically
converges in $2$ steps for $E_2$ and in $3$ steps for $E_3$.
Figure~\ref{fig:cantor} shows the sets $E_2, E_3$, the corresponding sets $L$ 
and the conformal map $\Phi$, evaluated numerically on a grid.
\end{example}

\begin{remark}
In all previous examples of sets $E = E_1 \cup \ldots \cup E_\ell$ consisting 
of $\ell$ real intervals, the centers $a_j$ were located in the intervals $E_j$.
This is not true in general, as the example $E = \cc{-1, 1} \cup \cc{1.2, 1.4}$
shows, where the computed centers of $L$ are
$a_1 = -0.0677\ldots \in E_1$, $a_2 = 1.0862\ldots \notin E_2$.
\end{remark}

\appendix
\section{Proofs}
\label{sect:appendix}

\begin{proof}[Proof of Corollary~\ref{cor:sum_mj_aj_poly_preimage}]
Let $\alpha$ be the coefficient of $z^{-2}$ in the Laurent series of $\partial_z g_E$ as in Lemma~\ref{lem:expansion_gE}.
We prove that $\alpha = - p_{n-1} / (n p_n)$.
Let $\cR : \comp{\Omega} \to \{ w \in \widehat{\C} : \abs{w} > 1 \}$ with 
$\cR(\infty) = \infty$ and $\cR'(\infty) > 0$
be the exterior Riemann map of $\Omega$.  Its Laurent series at infinity has
the form $\cR(z) = d_1 z + d_0 + \cO(z^{-1})$.
By~\cite[Eq.~(3.7)]{SchiefermayrSete2023},
\begin{equation*}
2 \partial_z g_E(z)
= \frac{1}{n} \frac{(\cR \circ P)'(z)}{(\cR \circ P)(z)}.
\end{equation*}
Since $n \geq 2$, we have
\begin{align*}
(\cR \circ P)(z) &= d_1 (p_n z^n + p_{n-1} z^{n-1} + \cO(z^{n-2})), \\
(\cR \circ P)'(z) &= d_1 (n p_n z^{n-1} + (n-1) p_{n-1} z^{n-2} + 
\cO(z^{n-3})),
\end{align*}
and hence
\begin{align*}
2 \partial_z g_E(z)
&= \frac{1}{n} \frac{n p_n z^{n-1} + (n-1) p_{n-1} z^{n-2} + \cO(z^{n-3})}{p_n 
z^n + p_{n-1} z^{n-1} + \cO(z^{n-2})} \\
&= \frac{z^{-1} + \frac{n-1}{n} \frac{p_{n-1}}{p_n} z^{-2} + \cO(z^{-3})}{1 + 
\frac{p_{n-1}}{p_n} z^{-1} + \cO(z^{-2})}.
\end{align*}
Using $1/(1 + z) = 1 - z + \cO(z^2)$ for $\abs{z} < 1$, we obtain
\begin{align*}
2 \partial_z g_E(z)
&= \Bigl( z^{-1} + \frac{n-1}{n} \frac{p_{n-1}}{p_n} z^{-2} + \cO(z^{-3}) \Bigr)
\Bigl( 1 -  \frac{p_{n-1}}{p_n} z^{-1} + \cO(z^{-2}) \Bigr) \\
&= z^{-1} - \frac{p_{n-1}}{n p_n} z^{-2} + \cO(z^{-3})
\end{align*}
and~\eqref{eqn:sum_mj_aj_poly_preimage} follows by applying Theorem~\ref{thm:sum_mj_aj}.
\end{proof}

\begin{proof}[Proof of Theorem~\ref{thm:mj_muE_intervals}]
We use the integral representation~\eqref{eqn:mj_with_integral}.
The Green's function $g_E$ is given by~\eqref{eqn:gE_l_intervals}, hence
the Wirtinger derivative of $g_E$ is
\begin{equation*}
2 \partial_z g_E(z) = \frac{R(z)}{\sqrt{H(z)}}
\end{equation*}
by~\eqref{eqn:wirtinger}, and therefore
\begin{equation*}
m_j = \frac{1}{2 \pi \ii} \int_{\gamma_j} \frac{R(z)}{\sqrt{H(z)}} \, \dd z,
\quad j = 1, \ldots, \ell,
\end{equation*}
where $\gamma_j$ is a (smooth) closed curve in $\C \setminus E$ with 
$\wind(\gamma_j; z) = \delta_{jk}$ for $z \in E_k$.
Since $R(z)/\sqrt{H(z)}$ is analytic in $\C \setminus E$ with branch cuts along 
the intervals $E_j$ (compare Lemma~\ref{lem:sqrtH}) and has singularities at the endpoints of the intervals, 
i.e., at $b_1, \ldots, b_{2n}$, we can deform $\gamma_j$ as follows without 
changing the value of the integral:
\begin{align*}
m_j &= \frac{1}{2 \pi \ii} \int_{b_{2j-1} + \eps}^{b_{2j} - \eps} 
\frac{R(x)}{\sqrt{H(x - 0 \ii)}} \, \dd x
- \frac{1}{2 \pi \ii} \int_{b_{2j-1} + \eps}^{b_{2j} - \eps} 
\frac{R(x)}{\sqrt{H(x + 0 \ii)}} \, \dd x \\
&\phantom{=} + \int_{C_1} \frac{R(z)}{\sqrt{H(z)}} \, \dd z
+ \int_{C_2} \frac{R(z)}{\sqrt{H(z)}} \, \dd z,
\end{align*}
where $\sqrt{H(x \pm 0 \ii)} = \lim_{y \searrow 0} \sqrt{H(x \pm \ii y)}$ 
denote the values of $\sqrt{H(z)}$ on the two rims of the branch cut 
$\cc{b_{2j-1}, b_{2j}}$, and where $C_1$, $C_2$ are circular arcs parametrized 
by $b_{2j-1} + \eps \ee^{\ii t}$, $0 \leq t \leq 2 \pi$, and $b_{2j} + \eps 
\ee^{\ii t}$, $-\pi \leq t \leq \pi$, respectively.
The integrals over the circular arcs vanish for $\eps \to 0$, since
\begin{equation*}
\abs*{\int_{C_k} \frac{R(z)}{\sqrt{H(z)}} \, \dd z} \leq L(C_k) \max_{z \in 
C_k} \frac{\abs{R(z)}}{\sqrt{\abs{H(z)}}} \leq 2 \pi \eps \frac{M}{\sqrt{\eps}},
\end{equation*}
which yields
\begin{equation*}
m_j = \frac{1}{2 \pi \ii} \int_{b_{2j-1}}^{b_{2j}} 
\Biggl( \frac{R(x)}{\sqrt{H(x - 0 \ii)}} - 
\frac{R(x)}{\sqrt{H(x + 0 \ii)}} \Biggr) \, \dd x.
\end{equation*}
By Lemma~\ref{lem:sqrtH} and together with~\eqref{eqn:sign_R}, we obtain for 
$x \in \oo{b_{2j-1}, b_{2j}}$, $j = 1, \ldots, \ell$,
\begin{equation*}
\frac{R(x)}{\sqrt{H(x - 0 \ii)}} - \frac{R(x)}{\sqrt{H(x + 0 \ii)}}
= \frac{(-1)^{\ell-j} \abs{R(x)}}{\sqrt{\abs{H(x)}}} \left( \frac{1}{-\ii 
(-1)^{\ell-j}} - \frac{1}{\ii (-1)^{\ell-j}} \right)
= 2 \ii \frac{\abs{R(x)}}{\sqrt{\abs{H(x)}}},
\end{equation*}
which is also stated in~\cite[Proof of Lem.~2.2\,(a)]{Peherstorfer1990}. 
This shows
\begin{equation*}
m_j = \frac{1}{\pi} \int_{b_{2j-1}}^{b_{2j}} 
\frac{\abs{R(x)}}{\sqrt{\abs{H(x)}}} \, \dd x
% = \mu_E(\cc{b_{2j-1}, b_{2j}}),
\end{equation*}
and the assertion now follows together 
with~\eqref{eqn:equilibrium_measure_one_interval}.
\end{proof}

\bibliographystyle{siam}
\bibliography{walshmap.bib}

\end{document}